\documentclass[10pt,a4paper]{amsart}

\usepackage[T1]{fontenc}	
\usepackage{graphicx}
\usepackage{microtype}
\usepackage{amsmath, amsthm, amssymb,centernot}
\usepackage{mathrsfs}
\usepackage{pdfsync, comment}
\usepackage[a4paper,top=3cm,bottom=3cm,left=3cm,right=3cm, bindingoffset=0mm]{geometry}

\usepackage[usenames,dvipsnames]{xcolor}
\usepackage[inline,shortlabels]{enumitem}
\usepackage[sort,numbers]{natbib}								
\usepackage[hyphens]{url}									
\usepackage{dsfont}									
\usepackage{stmaryrd}
\usepackage{upgreek}

\usepackage[
						colorlinks,				
						breaklinks,unicode,
						citecolor=OliveGreen,
						linkcolor=Maroon
						]{hyperref} 

\usepackage[font=footnotesize]{caption}
\usepackage[labelformat=simple,font=footnotesize]{subcaption}

\usepackage{mathtools}
\usepackage{scrextend}						
\usepackage{xparse}

\usepackage{tikz}
\usetikzlibrary{calc, positioning, arrows.meta,math,decorations.markings}

\newcommand{\boldnabla}{{\boldsymbol\nabla}}
\newcommand{\boldOmega}{{\boldsymbol\Omega}}

\ExplSyntaxOn
\NewDocumentCommand{\makeabbrev}{mmm}
 {
  \yoruk_makeabbrev:nnn { #1 } { #2 } { #3 }
 }

\cs_new_protected:Npn \yoruk_makeabbrev:nnn #1 #2 #3
 {
  \clist_map_inline:nn { #3 }
   {
    \cs_new_protected:cpn { #2 } { #1 { ##1 } }
   }
 }
 \ExplSyntaxOff

\makeabbrev{\textbf}{tbf#1}{a,b,c,d,e,f,g,h,i,j,k,l,m,n,o,p,q,r,s,t,u,v,w,x,y,z,A,B,C,D,E,F,G,H,I,J,K,L,M,N,O,P,Q,R,S,T,U,V,W,X,Y,Z}

\makeabbrev{\textbf}{bf#1}{a,b,c,d,e,f,g,h,i,j,k,l,m,n,o,p,q,r,s,t,u,v,w,x,y,z,A,B,C,D,E,F,G,H,I,J,K,L,M,N,O,P,Q,R,S,T,U,V,W,X,Y,Z}

\makeabbrev{\textsf}{tsf#1}{a,b,c,d,e,f,g,h,i,j,k,l,m,n,o,p,q,r,s,t,u,v,w,x,y,z,A,B,C,D,E,F,G,H,I,J,K,L,M,N,O,P,Q,R,S,T,U,V,W,X,Y,Z}

\makeabbrev{\mathsf}{mss#1}{a,b,c,d,e,f,g,h,i,j,k,l,m,n,o,p,q,r,s,t,u,v,w,x,y,z,A,B,C,D,E,F,G,H,I,J,K,L,M,N,O,P,Q,R,S,T,U,V,W,X,Y,Z}

\makeabbrev{\mathfrak}{mf#1}{a,b,c,d,e,f,g,h,i,j,k,l,m,n,o,p,q,r,s,t,u,v,w,x,y,z,A,B,C,D,E,F,G,H,I,J,K,L,M,N,O,P,Q,R,S,T,U,V,W,X,Y,Z}

\makeabbrev{\mathrm}{mrm#1}{a,b,c,d,e,f,g,h,i,j,k,l,m,n,o,p,q,r,s,t,u,v,w,x,y,z,A,B,C,D,E,F,G,H,I,J,K,L,M,N,O,P,Q,R,S,T,U,V,W,X,Y,Z}

\makeabbrev{\mathbf}{mbf#1}{a,b,c,d,e,f,g,h,i,j,k,l,m,n,o,p,q,r,s,t,u,v,w,x,y,z,A,B,C,D,E,F,G,H,I,J,K,L,M,N,O,P,Q,R,S,T,U,V,W,X,Y,Z}

\makeabbrev{\mathcal}{mc#1}{A,B,C,D,E,F,G,H,I,J,K,L,M,N,O,P,Q,R,S,T,U,V,W,X,Y,Z}

\makeabbrev{\mathbb}{mbb#1}{A,B,C,D,E,F,G,H,I,J,K,L,M,N,O,P,Q,R,S,T,U,V,W,X,Y,Z}

\makeabbrev{\mathscr}{ms#1}{A,B,C,D,E,F,G,H,I,J,K,L,M,N,O,P,Q,R,S,T,U,V,W,X,Y,Z}

\makeabbrev{\mathrm}{#1}{
Id,id,ran,rk,diag,stab,ann,conv,pr,ev,tr,End,Hom,sgn,im,op,can,fin,ext,red,tot,step,
rot,usc,lsc,Lip,LocLip,lip,bSymLip,osc,AC,loc,spec,coz,z,Hess,
supp,Opt,Adm,Cpl,Geo,geo,OptGeo,GeoAdm,GeoCpl,GeoSel,reg,
bd,co,Ric,Exp,dExp,dist,seg,Seg,cut,fcut,Cut,SDiff,Iso,Isom,diam,cl,Homeo,Diff,Der,der,grad,vol,dvol,inj,relint, Graph, sub,
var,law,Var,Poi,Gam,pa,so,iso,fs,inv,pqi,mix,
TestF,
}

\makeabbrev{\mathsf}{#1}{DP,CD,QCD,RQCD,BE,MCP,Ent,wMTW,MTW,RCD,EVI,Irr,IH,SC,wFe,UP}

\makeabbrev{\mathsc}{#1}{mmaf,cg,cc,lc}

\newcommand{\pbracket}[2]{\left\langle#1,#2\right\rangle}
\newcommand{\bbracket}[2]{\left\llbracket#1,#2\right\rrbracket}

\newcommand{\eps}{\varepsilon}
\renewcommand{\div}{\mathrm{div}}

\newcommand{\mathsc}[1]{\text{\textsc{#1}}}
\newcommand{\emparg}{{\,\cdot\,}}

\DeclareMathOperator{\eqdef}{\coloneqq}

\let\epsilon\varepsilon
\newcommand{\longrar}{\longrightarrow}

\newcommand{\diff}{\mathop{}\!\mathrm{d}}						
\newcommand{\ttabs}[1]{\lvert#1\rvert}	
	
\newcommand{\ttnorm}[1]{\lVert#1\rVert}
\newcommand{\norm}[1]{\left\lVert#1\right\rVert}					
\newcommand{\set}[1]{\left\{#1\right\}}							
\newcommand{\paren}[1]{\left(#1\right)}							
\newcommand{\tparen}[1]{\big({#1}\big)}
\newcommand{\bracket}[1]{\left[#1\right]}							

\newcommand{\class}[2][]{\left[#2\right]_{#1}}						
\newcommand{\tclass}[2][]{\big [#2\big]_{#1}}						
\newcommand{\ttclass}[2][]{[#2]_{#1}}							

\newcommand{\ttscalar}[2]{\langle #1 \, |\, #2\rangle}			
\newcommand{\scalar}[2]{\left\langle #1 \,\middle |\, #2\right\rangle}		
\newcommand{\sym}[1]{{\scriptscriptstyle{(#1)}}}
\newcommand{\tym}[1]{{\scriptscriptstyle{\times #1}}}
\newcommand{\otym}[1]{{\scriptscriptstyle{\otimes #1}}}

\newcommand{\tperp}{{\scriptscriptstyle{\perp}}}

\newcommand{\SFF}{\mathrm{I\!I}}

\DeclareSymbolFont{symbolsC}{U}{pxsyc}{m}{n}
\SetSymbolFont{symbolsC}{bold}{U}{pxsyc}{bx}{n}
\DeclareFontSubstitution{U}{pxsyc}{m}{n}
\DeclareMathSymbol{\medcirc}{\mathbin}{symbolsC}{7}
\DeclareSymbolFont{symbolsZ}{OMS}{pxsy}{m}{n}
\SetSymbolFont{symbolsZ}{bold}{OMS}{pxsy}{bx}{n}
\DeclareFontSubstitution{OMS}{pxsy}{m}{n}

\newcommand{\seq}[1]{\paren{#1}}								

\newcommand{\Cb}{\msC_b}									
\newcommand{\Cc}{\msC_c}									
\newcommand{\Cinfty}{{\msC^{\infty}}}						

\newcommand{\pfwd}{\sharp}

\DeclareMathOperator{\car}{\mathds 1}
 
\newcommand{\N}{{\mathbb N}}
\newcommand{\R}{{\mathbb R}}

\newcommand{\restr}{\big\lvert}

\newcommand{\comma}{\,\,\mathrm{,}\;\,}

\newcommand{\fstop}{\,\,\mathrm{.}}

\newcommand{\intr}{{\mathrm{int}}}

\newcommand{\conn}{\boldnabla}

\newcommand{\FC}[2]{\msF\msC^{#1}_{#2}}
\newcommand{\VC}[2]{\msV\msC^{#1}_{#2}}
\newcommand{\OC}[3]{\Omega_{#1}\msC^{#2}_{#3}}

\newcommand{\trid}{\star}

\allowdisplaybreaks

\numberwithin{equation}{section}

\theoremstyle{plain}
\newtheorem{theorem}{Theorem}[section]
\newtheorem*{theorem*}{Theorem}
\newtheorem*{mtheorem*}{Main Theorem}
\newtheorem{proposition}[theorem]{Proposition}
\newtheorem*{prop*}{Proposition}
\newtheorem{lemma}[theorem]{Lemma}
\newtheorem{corollary}[theorem]{Corollary}

\theoremstyle{definition}
\newtheorem{definition}[theorem]{Definition}
\newtheorem*{defs*}{Definition}

\theoremstyle{remark}
\newtheorem{remark}[theorem]{Remark}
\newtheorem*{rem*}{Remark}
\newtheorem*{quest*}{Question}

\renewcommand{\paragraph}[1]{\medskip{\noindent\emph{#1}.\quad}}

\makeatletter   
\def\@fnsymbol#1{\ensuremath{\ifcase#1\or *\or \mathsection \or  \mathparagraph \or \dagger\or
    \ddagger \or \|\or **\or \dagger\dagger
   \or \ddagger\ddagger \else\@ctrerr\fi}}
 \makeatother 

\makeatletter
\newcommand\thankssymb[1]{\textsuperscript{\@fnsymbol{#1}}}
\makeatother

\makeatother

\makeatletter
\def\l@subsection{\@tocline{2}{0pt}{2.5pc}{5pc}{}}
\makeatother

\begin{document}
\title[A `global' Perspective on Wasserstein Differential Geometry]{A `global' Perspective on\\ the Differential Geometry of Wasserstein Spaces\thankssymb{5}}

\author[L.~Dello Schiavo]{Lorenzo Dello Schiavo\thankssymb{1}}
\address{Università degli Studi di Roma ``Tor Vergata'' --- Dipartimento di Matematica, Viale della Ricerca Scientifica 1, 00133 Rome, Italy}
\email{delloschiavo@mat.uniroma2.it}
\thanks{\thankssymb{1}Università degli Studi di Roma ``Tor Vergata'', e-mail: delloschiavo@mat.uniroma2.it. Orcid: 0000-0002-9881-6870}

\author[A.~Pinamonti]{Andrea Pinamonti\thankssymb{2}}
\address{Universit\`a di Trento --- Dipartimento di Matematica. Via Sommarive, 14 - 38123 Povo, Italy}
\email{andrea.pinamonti@unitn.it}
\thanks{\thankssymb{2}%
Università di Trento, e-mail: andrea.pinamonti@unitn.it. Orcid: 0000-0003-1000-8607}
\thanks{\thankssymb{5}%
This project was initiated during a visit of LDS at the University of Trento, sponsored by a Research in Pairs project of FBK-CIRM (Trento), whose excellent working conditions are warmly acknowledged.
LDS gratefully acknowledges funding from the Austrian Science Fund (FWF) through project \href{https://doi.org/10.55776/ESP208}{10.55776/ESP208}, and from the PRIN Department of Excellence MatMod@TOV (CUP: E83C23000330006). AP was partially supported by the GNAMPA project \emph{Variational, Geometric, and Analytic Perspectives on Regularity} (CUP: E53C25002010001); further support was provided by MUR and by the University of Trento through the PRIN project \emph{Regularity problems in Sub-Riemannian structures} (MUR-PRIN 2022, Project code: 2022F4F2LH). LDS is also grateful to Masha Gordina and Karl-Theodor Sturm for useful conversations on the geometry of Wasserstein spaces.}

\begin{abstract}
We develop a global differential calculus on the $L^2$-Wasserstein space over a closed Riemannian manifold, based on derivations of cylinder functions rather than on the standard pointwise approach. 
Within this framework we define some fundamental geometric tools.
In particular, we show that the Levi-Civita connection on the base manifold lifts to the unique torsion-free connection compatible with the `extended' Otto metric. The corresponding Riemann tensor is exactly the lift of the base Riemann tensor, which shows that ---~in this framework~--- the correction terms of the classical gradient formalism are not intrinsic curvature terms, but arise from the projection onto the measure-dependent gradient distribution.
This allows us to revisit with a global and purely differential approach the smooth computations by  J.~Lott, \emph{Comm.\ Math.\ Phys.}, 277(2):423–437, 2007, reaching partially different conclusions.
\end{abstract}


\maketitle

\date{\today}
\maketitle

\section{Introduction}
Let~$\msP$ be the space of probability measures on a closed Riemannian manifold~$M$, endowed with the $L^2$-Kantorovich--Rubinstein distance~\eqref{eq:KR}.
The resulting metric space~$\msP_2$ is called the Wasserstein space over~$M$.
Since the foundational works~\cite{Bre91,BenBre00,McC97,Ott01,McC01}, and many others, suitable infinite-dimensional analogues on~$\msP_2$ of standard geometric objects have been introduced, including: \emph{tangent spaces} and a \emph{metric tensor}~\cite{Ott01}, the \emph{Levi-Civita connection}~\cite{Lot07}, the \emph{exponential map}~\cite{Gig11}, etc. 
Indeed, it is often said that~$\msP_2$ admits a formal structure of `infinite-dimensional Riemannian manifold' inherited from the underlying manifold~$M$.

In this paper we offer a `global' perspective on smooth computations on $L^2$-Wasserstein spaces which is in a sense alternative to the seminal one by J.~Lott,~\cite{Lot07}, and reaches partially different conclusions.
For this reason, before stating our results, we first provide a heuristic justification of our approach.

\subsection{Heuristics}
On a differential manifold~$M$, (smooth, tangent) vector fields can be defined in several equivalent ways.
In a `pointwise' fashion, the tangent space at each point is identified with the linear space of derivations acting on germs of smooth functions at that point; a vector field is then a section of the tangent bundle.
In a `global' fashion, a vector field is a derivation acting on the space of smooth functions.
Both are equivalent to a `local' approach, defining vector fields in charts: since derivations form a sheaf, the locality property (relating to the notion of germ) gives compatibility between the pointwise definition and the local one, while the gluing property (relating to the notion of chart) gives compatibility between the local definition and the global one.

It was noted in~\cite[\S6.1]{LzDS19b} that here the formal analogy between Wasserstein spaces and Riemannian manifolds breaks down, and this globalization phenomenon does not hold.
Indeed, consider as a natural class of `smooth' functions on~$\msP$ the algebra~$\FC{\infty}{b}$ of \emph{cylinder functions} induced by smooth potential energies, see~\eqref{eq:FC}.
While the usual tangent space to each point~$\mu$ in~$\msP$ is (the $L^2_\mu$-closure of) the space~$\mfX^\infty_\grad$ of smooth \emph{gradient} vector fields on~$M$, the space of derivations acting on~$\FC{\infty}{b}$ contains the space~$\mfX^\infty$ of \emph{all} smooth vector fields on~$M$, and in particular vector fields which are not the gradient of any function.
Thus, the `pointwise' perspective ---the mainstream one in optimal transportation--- and the `global' perspective ---which has occasionally transpired for~$\msP$, e.g.~\cite{GanKimPac10,LzDS19b}, and is the standard one on configuration spaces, e.g.~\cite{AlbKonRoe98,RoeSch99}--- diverge:
They cannot be bridged in view of the known lack of local charts on~$\msP$.

Here, we take this `global' point of view and we introduce vector fields directly as a module of derivations over~$\FC{\infty}{b}$.
(Incidentally, this is the modern approach also in metric measure geometry; cf., e.g.,~\cite{Gig18}.)
A conceptual strength of our approach is that it neatly separates \emph{differential} objects from \emph{Riemannian} ones, as we now explain.
A natural \emph{differential} structure on~$\msP$ (vector fields, Lie bracket, etc.) ought to be inherited from the differential structure of~$M$ alone, independently of any reference Riemannian metric~$g$ on~$M$.
This is indeed the case for the differential calculus we propose here, since we only rely on~$\mfX^\infty$.
By contrast, any differential calculus on~$\msP$ based on~$\mfX^\infty_\grad$ ---as in the `pointwise' approach~\cite{Gig12,Lot07,DinFan21,dSGMagRodSMa24}--- \emph{depends} on the Riemannian metric~$g$, since so does~$\mfX^\infty_\grad$.
Indeed, it follows from~\cite[Thm.~22]{LesSch20} that the standard differential structure of~$\msP$ is not even invariant under a \emph{conformal} change of the reference metric~$g$ on~$M$.

\smallskip

In order to use vector fields induced by~$\mfX^\infty$, we consider the natural extension~$\mbfG^\der$ of Otto's metric tensor~$\mbfG^\grad$ to the \emph{pseudo}-tangent space~$T^\der_\mu\msP$~\cite{GanKimPac10}, into which the usual tangent space~$T^\grad_\mu\msP$ naturally embeds.
(Here and elsewhere, the superscript `der' stands for `derivation'.)
This prompts us to define differential operators and tensors which naturally extend those considered in~\cite{Lot07} to the pseudo-tangent bundle.

\subsection{Main results}
Let~$\msP^\infty$ be the subspace of all measures in~$\msP$ with smooth density with respect to the Riemannian volume bounded away from zero (and infinity).
Similarly to~\cite{Lot07}, we compute the Lie bracket, the Levi-Civita connection, and the Riemann tensor of~$\msP$.
As a first important distinction, while~\cite{Lot07} performs formal computations only at points~$\mu$ in~$\msP^\infty$, we perform rigorous computations at \emph{every} $\mu$ in~$\msP$.
In contrast to~\cite{Lot07} (and to its recent reinterpretations~\cite{DinFan21,dSGMagRodSMa24}), the most of our computations is purely algebraic in nature, and independent of the regularity of the application point~$\mu$.

As a second crucial distinction, \cite{Lot07} performs computations only on \emph{constant} vector fields.
(This restriction appears already in the operative definition of bracket of vector fields~\cite[Eqn.~(4.1)]{Lot07}.)
Here, we perform computations on a larger linear space~$\VC{\infty}{b}\eqdef \FC{\infty}{b}\otimes \mfX^\infty$ of `smooth' \emph{cylinder vector fields}.
This choice is both relevant and necessary to our `global' understanding, for it allows us to verify whether operators on vector fields are \emph{tensorial}\footnote{Recall that an operator acting on a vector bundle on~$M$ is \emph{tensorial} if and only if it is $\Cinfty$-linear ---~thus, in particular invariant under change of coordinates.
Since no differential charts exist on~$\msP$ in the usual sense, we say that an operator acting on~$\VC{\infty}{b}$ is tensorial if it is $\FC{\infty}{b}$-linear.
Again, this fits our `global' point of view.
It is not possible to verify whether an operator acting \emph{only} on constant vector fields is tensorial, since the product of a `smooth' function in $\FC{\infty}{b}$ with a constant vector field is not a constant vector field, and therefore the operator is not defined on it.
}, which is not possible when working only with constant vector fields.

\paragraph{Lie brackets}
We revisit the Lie bracket in~\cite{Lot07} and prove that ---beyond being a commutator--- it makes~$\VC{\infty}{b}$ into a Lie algebroid with identical anchor on~$\FC{\infty}{b}$ (Cor.~\ref{c:LieAlgebroid}), in complete analogy with the case of manifolds.

\paragraph{Connections}
In~\S\ref{ss:IntrinsicConnection} we define an \emph{intrinsic} affine connection~$\conn^\intr$ on~$\msP$, acting on~$\VC{\infty}{b}$, which is \emph{independent} of any holonomy on~$M$.
Furthermore, we show (see Thm.~\ref{t:LCConnection} and Fig.~\ref{fig:Torsion})

\begin{theorem}\label{t:MainIntro} The following assertions hold:
\begin{itemize}[wide]
\item each affine connection~$\nabla$ on~$M$ lifts to an $\End(\VC{\infty}{b})$-valued $1$-form~$\boldOmega^\nabla$;

\item $\conn^\nabla \eqdef \conn^\intr + \boldOmega^\nabla$ is an affine connection acting on~$\VC{\infty}{b}$;

\item the association~$\nabla\mapsto \conn^\nabla$ preserves metric compatibility and torsion-freeness;

\item the connection~$\conn^{\nabla^\lc}$ induced on~$\msP$ by the Levi--Civita connection~$\nabla^\lc$ on~$M$ is the (unique) $\mbfG^\der$-metric-compatible torsion-free connection on~$\msP$, i.e., it is the Levi-Civita connection of $(\VC{\infty}{b},\mbfG^\der)$.
\end{itemize}
\end{theorem}

Notably, the intrinsic connection~$\conn^\intr$ identically \emph{vanishes} on constant vector fields, and therefore it does not show up in the computations in~\cite{Lot07}, which effectively deals with~$\boldOmega^\nabla$.
In our language, the latter is however \emph{not} a connection, since it is tensorial (being a differential form) and therefore it does not satisfy the product rule~\eqref{eq:p:Connection:0}.
On the contrary, we verify the product rule for both~$\conn^\intr$ and~$\conn^\nabla$ for every connection~$\nabla$ on~$M$.

\paragraph{Curvature}
Again with our coordinate-free algebraic approach, we compute the Riemann curvature tensor acting on~$\VC{\infty}{b}$, and show that it is simply the lift of the Riemann curvature tensor on~$M$, \emph{without} the additional correction terms appearing in~\cite{Lot07}; see Thm.~\ref{t:RiemannTensor}.
We further prove all natural symmetries and Bianchi identities for the Riemann tensor; see Prop.~\ref{p:RiemannIdentities}. 
The absence of correction terms in our curvature formula, cf.~\eqref{eq:t:RiemannTensor:0},
should not be regarded as a merely computational simplification.
Rather, it allows to identify the source of the correction terms appearing in the usual Wasserstein calculus.
In the global differential calculus developed here,
the curvature of the Wasserstein space is exactly the lift of the curvature of the underlying manifold.
Thus the Riemann tensor does not depend on the regularity of the base measure, nor on the discontinuous variation of the projection $\Pi_\mu\colon T^\der_\mu\msP\to T^\grad_\mu\msP$.

The additional terms found in the pointwise/gradient formalism arise only
after restricting the global calculus to the gradient distribution~$T^\grad\msP$.
From our perspective they are instead extrinsic terms,
measuring the failure of the gradient distribution to be preserved by the
global Levi-Civita connection and, equivalently, its non-integrability
inside the larger derivation bundle~$T^\der\msP$. This distinction
is one of the main conceptual points of the paper: the global curvature is
intrinsic, whereas the correction terms belong to the
reduced geometry obtained after projection onto~$T^\grad\msP$.

\subsection{Comparison with \cite{DinFan21,dSGMagRodSMa24}}
It will be apparent in the following that some of the expressions we compute for operators acting on $\VC{\infty}{b}$ are simpler than the corresponding expressions in the literature for operators acting on sections of the usual tangent bundle~$T^\grad\msP$.
In particular, under our `global' point of view, the bracket of constant vector fields is again a constant vector field in the pseudo-tangent bundle~$T^\der\msP$, which is \emph{not} true in the `pointwise' approach for the bracket of two gradient-type vector fields regarded as an element of~$T^\grad\msP$.
(Why the latter is not a constant vector field on~$\msP$ is carefully explained in~\cite[\S3]{DinFan21}.)

This is no contradiction, since the `global' bracket we consider is an extension to~$T^\der\msP$ of the `pointwise' bracket in~$T^\grad\msP$: while the `global' bracket (of constant vector fields) remains constant, the `pointwise' bracket varies (in fact, \emph{not} in a continuous way) as a function of~$\mu$, simply being the $L^2_\mu$-orthogonal projection of the `global' one onto~$T^\grad_\mu\msP=(\ker\div_\mu)^\tperp$. (See Fig.~\ref{fig:Lift} for a pictorial representation.)
In particular, while in order to \emph{define} the `pointwise' bracket at~$\mu$ it is necessary for~$\mu$ to induce a proper divergence, (see~\cite[\S3 and Dfn.~1.2]{DinFan21}) the `global' bracket is naturally defined at \emph{every} point~$\mu$ in~$\msP$.
In fact, this simplification is not limited to the bracket, but also applies to the Levi-Civita connection and to the Riemann tensor.

{\footnotesize
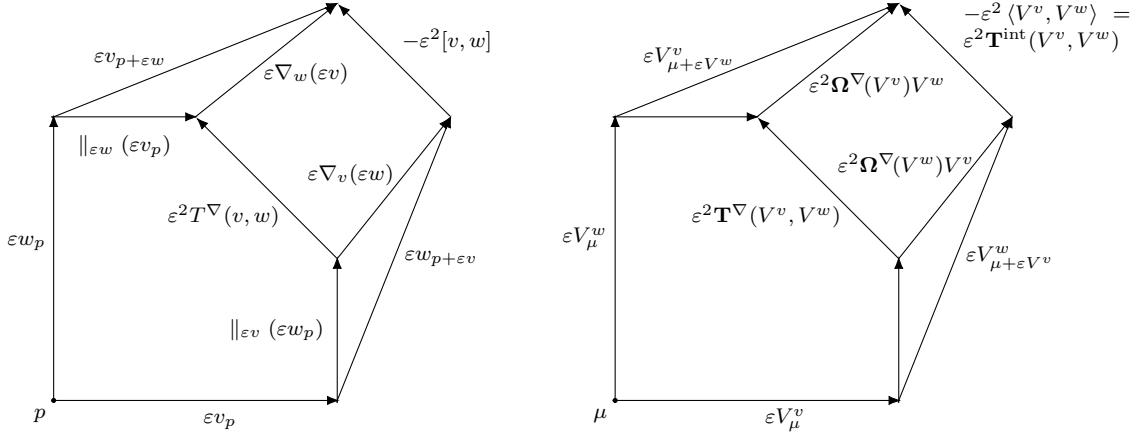
\begin{figure}[htb!]
\begin{tikzpicture}[scale=.75, >={Latex},
    shorten >=-4pt,
    shorten <=-4pt,]


\draw[fill] (0,0) circle (1pt) node[below left] {$p$};
\node (o) at (0,0) {};
\node (p) at (0,5) {};
\node (pw) at (2.5,5) {};
\node (q) at (5,0) {};
\node (qv) at (5,2.5) {};
\node (P) at (5,7) {};
\node (Q) at (7,5) {};

\draw [->] (o) -- (p) node[midway, above left] {$\eps w_p$};
\draw [->] (o) -- (q) node[midway, below right] {$\eps v_p$\phantom{$V^v_\mu$}};
\draw [->] (p) -- (pw) node[midway, below={4pt}] {$\parallel_{\eps w}(\eps v_p)$};
\draw [->] (q) -- (qv) node[midway, left={2pt}] {$\parallel_{\eps v}(\eps w_p)$};
\draw [->] (qv) -- (pw) node[midway, below left = 3pt and -8pt] {$\eps^2 T^\nabla(v,w)$};
\draw [->] (pw) -- (P) node[midway, below right = -2 pt and -4pt] {$\eps \nabla_w(\eps v)$};
\draw [->] (qv) -- (Q) node[midway, above left = -2pt and -3pt] {$\eps \nabla_v(\eps w)$};
\draw [->] (p) -- (P) node[midway, left={8pt}] {$\eps v_{p+\eps w}$};
\draw [->] (q) -- (Q) node[midway, right={1pt}] {$\eps w_{p+\eps v}$};
\draw [->] (Q) -- (P) node[midway, above right] {$-\eps^2[v,w]$};
\end{tikzpicture}
~\qquad\qquad~
\begin{tikzpicture}[scale=.75, >={Latex},
    shorten >=-4pt,
    shorten <=-4pt,]


\draw[fill] (0,0) circle (1pt) node[below left] {$\mu$};
\node (o) at (0,0) {};
\node (p) at (0,5) {};
\node (pw) at (2.5,5) {};
\node (q) at (5,0) {};
\node (qv) at (5,2.5) {};
\node (P) at (5,7) {};
\node (Q) at (7,5) {};

\draw [->] (o) -- (p) node[midway, above left] {$\eps V^w_\mu$};
\draw [->] (o) -- (q) node[midway, below right] {$\eps V^v_\mu$};
\draw [->] (p) -- (pw); 
\draw [->] (q) -- (qv); 
\draw [->] (qv) -- (pw) node[midway, below left = 3pt and -8pt] {$\eps^2 \mbfT^\nabla(V^v,V^w)$};
\draw [->] (pw) -- (P) node[midway, below right = 2pt and -10pt] {$\eps^2 \boldOmega^\nabla\!(V^v)V^w$};
\draw [->] (qv) -- (Q) node[midway, above left = 1pt and -10pt] {$\eps^2 \boldOmega^\nabla\!(V^w)V^v$};
\draw [->] (p) -- (P) node[midway, left= 5pt] {$\eps V^v_{\mu+\eps V^w}$};
\draw [->] (q) -- (Q) node[midway, right={1pt}] {$\eps V^w_{\mu+\eps V^v}$};
\draw [->] (Q) -- (P) node[text width=3cm, midway, above right] {$-\eps^2\pbracket{V^v}{V^w} =$ $\eps^2 \mbfT^\intr(V^v,V^w)$};
\end{tikzpicture}

\caption{(\emph{left}) A graphical representation on vector fields~$v,w$ at a point~$p$ of a connection~$\nabla$ on~$M$, of its torsion~$T^\nabla$, and of its parallel transport~$\parallel$, to be compared with their counterparts on~$\msP$ (right).
\\
\phantom{i}\quad (\emph{right}) A graphical representation on \emph{constant} vector fields~$V^v, V^w$ at a point~$\mu$ of the connection~$\conn^\nabla$ on~$\msP$ induced by a connection~$\nabla$ on~$M$, to be compared with their counterparts on~$M$ (left). (We denote by~$\mbfT^\nabla$ the torsion of the connection~$\conn^\nabla$ on~$\msP$.)
Since~$\conn^\intr$ vanishes on constant fields, the connection coincides with the $\End$-valued one-form, viz.~$\conn^\nabla_{V^w}V^v= \boldOmega^\nabla(V^v)V^w$ and analogously when exchanging $v$ and~$w$.
The torsion~$\mbfT^\intr$ of the intrinsic connection coincides with the negative pseudo-bracket of vector fields (Prop.~\ref{p:Torsion}), which, for constant fields, coincides in turn with the full bracket (Rmk.~\ref{r:BracketConstantFields}).
}\label{fig:Torsion}
\end{figure}
}

\begin{figure}[htb!]

\tikzset{
    mark position/.style args={#1(#2)}{
        postaction={
            decorate,
            decoration={
                markings,
                mark=at position #1 with \coordinate (#2);
            }
        }
    }
}

\begin{tikzpicture}[scale=1, >={Latex}, declare function={curva(\x)=ln(\x)+3*sin(12*(\x));}]
\footnotesize

\def\n{11};
\def\init{3};

\draw[smooth, dashed, samples=100, variable=\t, domain=\init:\n] plot ({\t}, {curva(\t)});

\foreach \x [evaluate=\x as \y using int((\x-\init)/2)] in {3,5,...,\n}
	{
		\draw[fill, black, opacity=\x/(3\n), cm={.4,0,0,.6,(\x,curva(\x))}] (-5,-4) -- (-4,5) -- (5,4) -- (4, -5) -- (-5,-4) node[black, opacity=\x/\n, above left=-3pt and 70pt] at (5,4) {$T^\der_{\mu_{t_{\y}}}\msP$};
		
		\begin{scope}
		\clip[cm={.4,0,0,.6,(\x,curva(\x))}] (-5,-4) -- (-4,5) -- (5,4) -- (4, -5) -- cycle;
		
		\draw[black, opacity=\x/\n, cm={cos(6*\x-3*pi),-sin(6*\x-3*pi),sin(6*\x-3*pi),cos(6*\x-3*pi),(\x,curva(\x))}, variable=\t, domain=-6:0, mark position=-.3(a)] plot ({\t}, 0);
		
		\draw[black, opacity=\x/\n, cm={cos(6*\x-3*pi),-sin(6*\x-3*pi),sin(6*\x-3*pi),cos(6*\x-3*pi),(\x,curva(\x))}, variable=\t, domain=1.5*cos(6*\x-3*pi)+.6:4] plot ({\t}, 0);

		\draw[->, black, opacity=\x/\n, cm={1,0,0,1,(\x,curva(\x))}] (0,0) -- (1.5,0) node[above left= 0pt and 5pt] {$
		w$};
		
		\draw[->, black, opacity=\x/\n, cm={cos(6*\x-3*pi)^2, -cos(6*\x-3*pi)* sin(6*\x-3*pi), -cos(6*\x-3*pi)* sin(6*\x-3*pi), sin(6*\x-3*pi)^2,(\x,curva(\x))}] (0,0) -- (1.5,0) node[below, opacity=\x/\n] {$\Pi_{\mu_{t_{\y}}}\! w$};
		\end{scope}
		
		\node at (a) [below left = 0pt and -5pt, black,opacity=\x/\n] {$T^\grad_{\mu_{t_{\y}}}\msP$};
	};
\end{tikzpicture}

\caption{A graphical representation of the projection~$\Pi_\mu$ onto the usual tangent space~$T^\grad_\mu\msP=(\ker\div_\mu)^\tperp$  (rendered as a line), a subspace of the pseudo-tangent space~$T^\der_\mu\msP$ (rendered as a plane).
For a vector field~$V^w=w$ constant along a curve~$\seq{\mu_t}$ (dashed), its $\div_{\mu_t}$-free part~$\Pi_{\mu_t}w$ \emph{changes} along the curve, since so does the way in which~$T^\grad_{\mu_t}\msP$ lies inside~$T^\der_{\mu_t}\msP$.
}
\label{fig:Lift}
\end{figure}
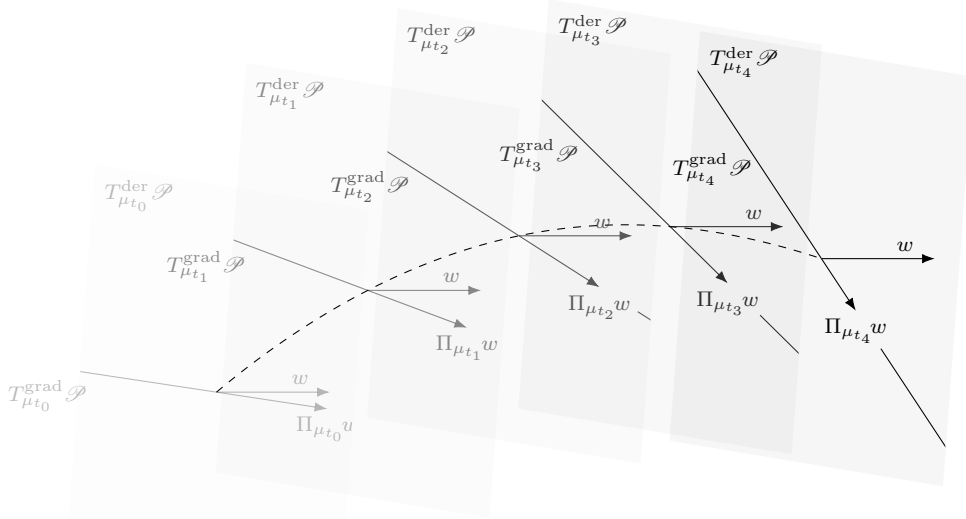

\section{Preliminaries}
Let~$M$ be a closed (i.e., compact boundaryless) smooth connected oriented differential manifold of dimension~$n$.
We denote by~$TM$, resp.~$T^*M$, the tangent, resp.\ cotangent bundle to~$M$.
We denote by~$\mfX^\infty$ the vector space of (smooth) vector fields (i.e., smooth sections of~$TM$), and, for~$k\in \N_0$, by~$\Omega^\infty_k$ the alternating algebra (with~$\wedge$) of (smooth) differential $k$-forms (i.e., smooth sections of the $k$-th exterior power of~$T^*M$).
For a differential form~$\omega\in\Omega^\infty_k$, we let~$\diff\omega\in\Omega^\infty_{k+1}$ be the exterior differential of~$\omega$.
Evaluation at a point is often denoted by a subscript and contracted over operators.
For example,~$\diff_x f=(\diff f)_x$ denotes the differential of a function~$f$ at a point~$x$ in~$M$.

\paragraph{Flows}
When endowed with the standard Lie bracket~$[\emparg, \emparg]$, the space~$\mfX^\infty$ is a Lie algebra.
For a vector field~$w\in\mfX^\infty$ we denote by~$\uppsi^{w,t}\colon M\to M$ the flow of~$w$, defined pointwise as the unique solution to the \textsc{ode}
\[
\diff_t \tparen{\uppsi^{w,t}(x)} = w\tparen{\uppsi^{w,t}(x)},\quad {\uppsi^{w,0}(x)=x}\fstop
\]
Since~$M$ is closed,~$\uppsi^{w,t}$ exists for each~$t\in\R$ and~$\seq{\uppsi^{w,t}}_{t}$ is a one-parameter subgroup of
\[
\Diff^\infty(M)_0\eqdef \{f\in \Diff^{\infty}(M) : f\ \mbox{is smoothly isotopic to the identity}\} \fstop
\]

\paragraph{Riemannian structure}
In the following, let~$g$ be any Riemannian metric on~$M$.
We denote by~$\mssd_g$ the intrinsic distance and by~$\vol_g$ the Riemannian volume measure induced on~$M$ by~$g$.
Since~$M$ is compact, we may and will always assume that~$g$ is normalized in such a way that~$\vol_g M=1$.
Further, let~$g^\sharp\colon T^*M\to TM$ be the musical isomorphism induced by~$g$ on~$T^*M$.
With a standard abuse of notation we also denote by~$g^\sharp$ the induced map on sections~$g^\sharp\colon \Omega^\infty_1\to\mfX^\infty$.
Finally, we denote by~$\mfX^\infty_\grad=g^\sharp \diff \Cinfty$ the space of gradient-type vector fields.
Note that~$\mfX^\infty_\grad$ depends on~$g$. This dependence is omitted from the notation for simplicity.

\subsection{The $W_2$-metric structure of~$\msP$}
Let~$\msP$ be the space of all Borel probability measures on~$M$.
The \emph{$L^2$-Kantorovich--Rubinstein} distance~$W_2$ on~$\msP$ is defined by
\begin{equation}\label{eq:KR}
W_2(\mu,\nu)^2 = \inf \int \mssd_g(x,y)^2\diff \pi(x,y)\comma \qquad \mu,\nu\in\msP\comma
\end{equation}
where the infimum is taken over all couplings~$\pi$ of~$\mu$ and~$\nu$.
Since~$M$ is a compact manifold, the \emph{$L^2$-Wasserstein space}~$\msP_2\eqdef(\msP,W_2)$ on~$M$ is a complete separable geodesic metric space, see e.g.~\cite[Ch.~6]{Vil09} and~\cite{AmbGig11}.

\subsection{The $W_2$-geometry of~$\msP$}
For each continuous real-valued function~$f$ on~$M$, denote by~$f^\trid$ the linear functional on~$\msP$ induced by integration w.r.t.~$f$, viz.
\[
f^\trid\colon \mu \longmapsto \mu f\eqdef \int f\diff\mu \fstop
\]
Each~$\mu\in\msP$ induces a pre-Hilbert semi-norm on~$\mfX^\infty$ by setting
\[
\norm{w}_\mu^2 \eqdef \mu\tparen{g(w,w)} \comma \qquad w\in\mfX^\infty \fstop
\]
We denote by~$\class[\mu]{\mfX^\infty}$ the quotient space of~$(\mfX^\infty,\norm{\emparg}_\mu)$ obtained by identifying elements with vanishing semi-norm~$\norm{\emparg}$, and again by~$\norm{\emparg}_\mu$ the pre-Hilbert norm on~{$\class[\mu]{\mfX^\infty}$.}
Finally, let~$\mfX_\mu$ be the Hilbert completion of~$(\class[\mu]{\mfX^\infty},\norm{\emparg}_\mu)$, and denote again by~$\norm{\emparg}_\mu$ the Hilbert norm on~$\mfX_\mu$. We define similarly $\tparen{\tclass[\mu]{\mfX^\infty_\grad},\norm{\emparg}_\mu}$.

\paragraph{Divergences}
Denote by~$\msD'$ the space of distributions on~$M$.
Each~$\mu\in\msP$ induces a \emph{divergence operator}~$\div_\mu\colon\mfX^\infty\to \msD'$ defined by
\[
w\longmapsto \Big( \scalar{\div_\mu w}{\emparg} \colon f \mapsto \mu \tparen{(\diff f)w} \Big)
\]
Since~$\scalar{\div_\mu w}{f}\leq \norm{\nabla f}_{\mu}\norm{w}_{\mu}$, then~$\div_\mu$ has a non-relabeled extension~$\div_\mu\colon \mfX_\mu \to \msD'$.
We note that our sign convention for the divergence~$\div_\mu$ may differ from the convention used elsewhere in the literature.

\paragraph{Tangent spaces}
At each~$\mu\in\msP$ we have three different tangent spaces:
\begin{itemize}
\item the (Hilbert) \emph{tangent space}~$T^\grad_\mu\msP\eqdef \cl_{\mfX_\mu}\tparen{\ttclass[\mu]{\mfX^\infty_\grad}}$;
\item the (Banach) \emph{geometric tangent space},~$T^\geo_\mu\msP$, introduced in~\cite[Dfn.~5.4]{Gig11};
\item the (Hilbert) \emph{pseudo-tangent space}~$T^\der_\mu\msP\eqdef \mfX_\mu$, considered as an auxiliary space in~\cite{GanKimPac10,ChoGan17}.
\end{itemize}

N.~Gigli introduced~$T^\geo_\mu\msP$ by showing that it coincides with an abstract completion of the space of germs of $W_2$-geodesics emanating from~$\mu$.
Furthermore, he completely characterized the relations between the spaces~$T^\grad_\mu\msP$ and~$T^\geo_\mu\msP$ in terms of \emph{transport-regularity}.
Recall that~$\mu$ is transport-regular if, for every real-valued semiconvex function~$\phi$ on~$M$, $\mu$~does not charge the singular set of~$\phi$; see~\cite[Dfn.~2.8]{Gig11}.

\begin{proposition}[Gigli,~{\cite[\S6]{Gig11}}]
$T^\grad_\mu\msP$ embeds isometrically into~$T^\geo_\mu\msP$ for every~$\mu\in\msP$.
Furthermore, the following are equivalent:
\begin{itemize}
\item $\mu$ is transport-regular;
\item $T^\grad_\mu\msP=T^\geo_\mu\msP$;
\item $T^\geo_\mu\msP$ is a Hilbert space.
\end{itemize}
\end{proposition}

While~$T^\geo_\mu\msP$ carries all the geometric information on (directions of) geodesics emanating from~$\mu$, the simpler space~$T^\grad_\mu\msP$ only carries information on those geodesics emanating from~$\mu$ that are induced by transport maps.
The pseudo-tangent space~$T^\der_\mu\msP$ contains additional lines which do not correspond to geodesics emanating from~$\mu$, constituting the subspace~$\ker \div_\mu$ in the \emph{Helmholtz decomposition} (e.g.~\cite[Rmk.~2.7]{GanKimPac10})
\[
T^\der_\mu = T^\grad_\mu\msP \oplus^\perp \ker \div_\mu \comma
\]
where~$\oplus^\perp$ denotes the orthogonal direct sum of Hilbert spaces.
In the following we will denote by~$\Pi_\mu$ the orthogonal projection onto~$T^\grad_\mu\msP$, viz.
\begin{equation}\label{eq:Projection}
\Pi_\mu\colon T^\der_\mu \longrar T^\grad_\mu\msP \fstop
\end{equation}

\paragraph{Lack of smooth charts}
We stress that none of the above objects varies continuously in~$\mu$ on the whole of~$\msP$.
For instance, whenever~$\mu=\sum_i s_i\delta_{x_i}$ is a finite convex combination of Dirac masses, we have (e.g.~\cite[Ex.~2.8]{GanKimPac10})
\[
\tparen{T^\grad_\mu,\norm{\emparg}_\mu} = \tparen{T^\der_\mu,\norm{\emparg}_\mu} \cong \bigoplus_i \tparen{T_{x_i}M, s_i\sqrt{g_{x_i}}} \comma\qquad \ker\div_\mu=\set{0} \fstop
\]
In particular, when~$\mu$ is supported on~$k$ points,~$\dim T^\grad_\mu\msP = k\dim M$, whereas when~$\mu\ll \vol_g$ we have~$\dim T^\grad_\mu\msP=\infty$.
Since the set of finite convex combinations of Dirac masses is dense in~$\msP$, it follows that the dimension of~$T^\grad_\mu\msP$ is not constant in any neighborhood of every point. In particular,~$\msP$ is not a manifold modelled on any Hilbert, Banach, or Fr\'echet space.

It is nonetheless possible, as noted by J.~Lott in~\cite{Lot07}, to give a Fréchet manifold structure in the sense of~\cite{KriMic97} to the subset
\[
\msP^\infty \eqdef\set{\mu\in\msP : \mu = \rho\, \vol_g\comma \rho\in\Cinfty\comma \rho>0}
\]
of all absolutely continuous measures with smooth nowhere vanishing density.

\subsection{Cylinder functions and vector fields}

\paragraph{Cylinder functions}
Denote by~$\FC{\infty}{b}$ the set of \emph{cylinder functions} on~$\msP$ induced by~$\Cinfty$, viz.
\begin{equation}\label{eq:FC}
\FC{\infty}{b}\eqdef \set{u\colon\msP \to \R : \begin{gathered} u=F\circ\mbff^\trid , k\in\N_0, F\in \Cinfty(\R^k), \\ \mbff=\seq{f_i}_{i\leq k}, f_i\in\Cinfty\quad i\leq k \end{gathered}} \fstop
\end{equation}

\begin{remark}[See e.g.~\cite{LzDS19b}]
\begin{enumerate*}[$(a)$]
\item The representation of~$u\in \FC{\infty}{b}$ by~$F$ and~$\mbff$ is in general not unique;
\item $\FC{\infty}{b}$ is an algebra;
\item Choosing~$k=0$ in~\eqref{eq:FC} shows that~$\FC{\infty}{b}$ is unital.
\item Since~$f^\trid$ is bounded on~$\msP$ for every~$f\in \Cb$, in the definition of~$\FC{\infty}{b}$ we may equivalently choose~$F\in\Cc^\infty(\R^k)$ or~$F\in\Cinfty(\R^k)$.
\end{enumerate*}
\end{remark}

By the next lemma, cylinder functions are sufficient for all topological purposes.

\begin{lemma}
$\FC{\infty}{b}$ is dense in~$\Cb(\msP)$.
\begin{proof}
Since~$M$ is compact,~$\msP$ is compact too.
The proof follows by the Stone--Weierstra\ss\ Theorem.
\end{proof}
\end{lemma}

\begin{definition}[Differential of cylinder functions]
We define the (\emph{exterior}) \emph{differential}~$\mbfD u$ of a cylinder function~$u\in\FC{\infty}{b}$ as the ($\mu$-class of the) $1$-form
\[
\mbfD_\mu u = (\mbfD u)_\mu \colon x\longmapsto \sum_i (\partial_i F)(\mbff^\trid\mu) \diff_x f_i \comma \qquad \mu\in\msP\fstop
\]
\begin{proof}
We need to show that the definition is well-posed, i.e.\ that~$\mbfD_\mu u$ is independent of the choice of representatives~$u=F\circ \mbff^\trid$.
This is the case since, for \emph{every}~$w\in\mfX^\infty$, we have (cf.\ e.g.~\cite[Eqn.~(2.5) and Lem.~6.2]{LzDS19b})
\[
\mu\tparen{(\mbfD_\mu u)w}= \int (\mbfD_\mu u)_x w_x \diff\mu(x) = \diff_t\restr_{t=0} u\tparen{\uppsi^{w,t}_\pfwd\mu}\comma \qquad \mu\in\msP\fstop \qedhere
\]
\end{proof}
\end{definition}

\paragraph{Cylinder differential forms and cylinder vector fields}
Denote by~$\otimes$ the algebraic $\R$-tensor product.

\begin{definition}
We define the following linear spaces
\begin{itemize}
\item $\OC{k}{\infty}{b}\eqdef \FC{\infty}{b} \otimes \Omega^\infty_k$, the space of \emph{cylinder $k$-forms}, for every~$k\in\N_0$;

\item $\VC{\infty}{b}\eqdef \FC{\infty}{b} \otimes \mfX^\infty$, the space of \emph{cylinder vector fields};

\item $\VC{\infty}{b,\grad}\eqdef \FC{\infty}{b} \otimes \nabla^g\Cinfty$, the space of \emph{cylinder gradient-type vector fields}.
\end{itemize}
\end{definition}

The first definition is motivated by the fact that
\begin{equation}\label{eq:ExactCylinderForms}
\mbfD \FC{\infty}{b} \subset \FC{\infty}{b} \otimes \diff \Cinfty \fstop
\end{equation}

As for the second definition, let~$\boldsymbol{g}^\sharp\colon \OC{1}{\infty}{b}\to \VC{\infty}{b}$ be the $\FC{\infty}{b}$-linear extension of~$g^\sharp\colon \Omega^\infty_1\to\mfX^\infty$ to~$\OC{1}{\infty}{b}$.
Applying~$\boldsymbol{g}^\sharp$ on both sides of~\eqref{eq:ExactCylinderForms}, and noting that~$\boldsymbol{g}^\sharp\mbfD=\boldnabla^g$ coincides on~$\FC{\infty}{b}$ with the Wasserstein gradient on~$\msP_2(M,\mssd_g)$, while~$g^\sharp\diff=\nabla^g$ is the $g$-induced gradient on~$M$,
\[
\boldnabla^g \FC{\infty}{b} \subset \FC{\infty}{b} \otimes \nabla^g \Cinfty \fstop
\]

The application of~$\OC{k}{\infty}{b}$ to~${\VC{\infty}{b}}^\otym{k}$ is defined as the $\FC{\infty}{b}$-multilinear extension of the application of~$\Omega^\infty_k$ to~${\mfX^\infty}^\otym{k}$.
Explicitly, for a cylinder $k$-form~$\Omega = \sum_i u_i\otimes \omega_i \in \OC{k}{\infty}{b}$ and cylinder vector fields~$V^j\eqdef \sum_{i_j} u^j_{i_j} \otimes w^j_{i_j} \in \VC{\infty}{b}$, with $j\leq k$, we have
\begin{equation}\label{eq:FormAction}
\Omega(V^1,\dotsc, V^k)\eqdef \sum_i \sum_{i_1,\dotsc, i_k} u_i u_{i_1}\cdots u_{i_k} \cdot \tparen{\omega_i(w^1_{i_1}, \dotsc, w^k_{i_k})}^\trid \fstop
\end{equation}

\begin{definition}[Vector fields on functions, covariant derivatives]
Let~$V \in\VC{\infty}{b}$ and~$u\in\FC{\infty}{b}$. We define the application of~$V$ to~$u$ as the application to~$V$ of the $1$-form~$\mbfD u$, that is, as the cylinder function
\begin{equation}\label{eq:VFAction}
Vu \colon \mu\longmapsto \mu \tparen{(\mbfD_\mu u)V_\mu} = \int (\mbfD_\mu u)_x (V_\mu)_x \diff\mu(x) \fstop
\end{equation}

By analogy with the case of standard differential manifolds, we may also interpret~$Vu$ as the \emph{covariant derivative}~$\conn_V u$.
\end{definition}

\begin{remark}[Identification modulo~$\mu$]
Properly speaking, no~$V\in\VC{\infty}{b}$ is a section of~$T^\der\msP$, since in fact we always ought to regard~$V$ at~$\mu$ as its class modulo $\mu$-negligible sets.
Already for constant vector fields~$V^w\eqdef\car\otimes w$, with~$w\in\mfX^\infty$, this may cause a substantial reduction of the information carried by the vector field.
For example, in the extreme case when~$\mu=\delta_{x_0}$, the information carried by~$V^w$ at~$\mu$ ---that is, the whole vector field~$x\mapsto w_x$--- trivializes to the single vector~$w_{x_0}\in T_{x_0}M$.

In what follows, we will always disregard this issue, never introducing a notation for classes modulo $\mu$-negligible sets.
Our computations remain nonetheless rigorous at \emph{every} point~$\mu\in\msP$, since the very definition of the action~\eqref{eq:FormAction} of a $k$-form on a $k$-tuple of vector fields contains an integration w.r.t.\ the application point~$\mu$.
The same applies to the action~\eqref{eq:VFAction} of vector fields on functions, which is a particular case of~\eqref{eq:FormAction}.
In other words, the integration w.r.t.~$\mu$ in~\eqref{eq:VFAction} automatically takes care of the fact that~$x\mapsto (V_\mu)_x$ must be regarded modulo $\mu$-negligible sets.
\end{remark}

\begin{remark}[Identification modulo $\ker\div_\mu$]
By the very definition~\eqref{eq:VFAction} of the action of~$\VC{\infty}{b}$ on~$\FC{\infty}{b}$, we may \emph{equivalently} regard every~$V\in \VC{\infty}{b}$ as a section of the tangent bundle~$T^\grad\msP$, rather than of a section of the pseudo-tangent bundle~$T^\der\msP$.
Indeed, since~$\mbfD u$ is a section of~$T^\grad\msP$ for every~$u\in\FC{\infty}{b}$ by definition of~$\mbfD$,
\[
(Vu)_\mu = \mu\tparen{(\mbfD_\mu u)V_\mu} =  \mu\paren{\tparen{\Pi_\mu(\mbfD_\mu u)}V_\mu} =  \mu\tparen{(\mbfD_\mu u)(\Pi_\mu V_\mu)} = \tparen{(\Pi V)u}_\mu\comma
\]
and therefore
\begin{equation}\label{eq:KernelPi}
Vu = (\Pi V)u \comma \qquad V\in\VC{\infty}{b}\comma u\in\FC{\infty}{b}\fstop
\end{equation}
\end{remark}

\begin{remark}[Vector fields as `global' derivations]
It is shown in~\cite[Prop.~6.3]{LzDS19b} that the map~$\partial\colon V\mapsto \boldnabla_V$ is a linear injection of~$\VC{\infty}{b}$ into the space of derivations~$\Der(\FC{\infty}{b})$.
In particular, the image via $\partial$ of~$\VC{\infty}{b,\grad}$ is a \emph{proper} subspace of~$\Der(\FC{\infty}{b})$.
This does not contradict~\eqref{eq:KernelPi}, since for at least some~$\mu\in\msP$ we have~$\ker\div_\mu=\set{0}$.
In particular, it is not true that~$V_\emparg = \Pi_{\mu_o} V_\emparg$ for some fixed~$\mu_o$, not even for~$\mu_o= \vol_g$.
\end{remark}

\subsection{Intrinsic connection and brackets}\label{ss:IntrinsicConnection}
We define an operator~$\conn^\intr$ as the $\mfX^\infty$-linear extension to~$\VC{\infty}{b}$ of the differential~$\mbfD\colon \FC{\infty}{b}\to \OC{1}{\infty}{b}$.
Precisely,~$\conn^\intr$ is the linear extension of the operator on elementary tensors~$u\otimes v$ defined as
\begin{align*}
\conn^\intr \colon \VC{\infty}{b} &\longrar \FC{\infty}{b} \otimes (\Omega^\infty_1\times \mfX^\infty)
\end{align*}
satisfying, for every~$(V,\Omega)\in \VC{\infty}{b} \otimes \OC{1}{\infty}{b}$ and every~$u\in\FC{\infty}{b}$,
\begin{equation}\label{eq:Connection}
\paren{\tparen{\conn^\intr _\mu (u\otimes w)} (V,\Omega)}_\mu \eqdef \iint \tparen{(\mbfD_\mu u) V_\mu}_x \cdot (\Omega_\mu w)_y \diff\mu^\otym{2}(x,y)\fstop
\end{equation}

\begin{proposition}[Intrinsic connection]\label{p:Connection}
The operator~$\conn^\intr$ is a connection on~$\VC{\infty}{b}$, that is, it is $\R$-linear and satisfies the product rule
\begin{align}\label{eq:p:Connection:0}
\conn^\intr(vZ) = v \conn^\intr Z + \mbfD v \otimes Z\comma \qquad v\in\FC{\infty}{b}\comma \quad Z\in\VC{\infty}{b} \fstop
\end{align}

\begin{proof}
$\R$-linearity is straightforward, so we only verify the product rule~\eqref{eq:p:Connection:0}.
We need to show that, for every~$V, Z\in\VC{\infty}{b}$, every~$\Omega\in \OC{1}{\infty}{b}$, and every~$v\in\FC{\infty}{b}$,
\begin{align*}
\paren{\tparen{\conn^\intr _\mu (vZ)} (V,\Omega)}_\mu = v_\mu \cdot\paren{\tparen{\conn^\intr _\mu Z} (V,\Omega)}_\mu + (Vv)_\mu \cdot (\Omega Z)_\mu \fstop
\end{align*}
By $\R$-linearity of~$Z\mapsto \paren{\tparen{\conn^\intr _\mu Z} (V,\Omega)}_\mu$ it suffices to prove the above equality on elementary tensors~$Z=u\otimes w$ with~$u\in\FC{\infty}{b}$.
We have
\begin{align*}
&\paren{\tparen{\conn^\intr _\mu (uv\otimes w)} (V,\Omega)}_\mu
\\
&= \iint \tparen{(\mbfD_\mu u) v_\mu V_\mu}_x \cdot (\Omega_\mu w)_y \diff\mu^\otym{2}(x,y)
+\iint \tparen{u_\mu\, (\mbfD_\mu v) V_\mu}_x \cdot (\Omega_\mu w)_y \diff\mu^\otym{2}(x,y)
\\
&= v_\mu \iint \tparen{(\mbfD_\mu u) V_\mu}_x \cdot (\Omega_\mu w)_y \diff\mu^\otym{2}(x,y)
+\int \tparen{(\mbfD_\mu v) V_\mu}_x \diff\mu(x) \cdot \int u_\mu\, (\Omega_\mu w)_y \diff\mu(y)
\\
&= \paren{v\tparen{\conn^\intr _\mu (u\otimes w)} (V,\Omega)}_\mu + (Vv)_\mu \cdot \tparen{\Omega (u\otimes w)}_\mu \fstop \qedhere
\end{align*}
\end{proof}
\end{proposition}

It follows from the proposition that setting
\[
\tparen{(\conn^\intr_VZ) u}_\mu \eqdef \tparen{(\conn^\intr_\mu Z)(V,\mbfD_\emparg u)}_\mu
\]
defines a \emph{covariant derivative} of~$Z$ along~$V$ computed on a function~$u$ at the point~$\mu$.

We call~$\conn^\intr$ the \emph{intrinsic connection on~$\msP$}.
Indeed,~$\conn^\intr$ `does not see' the base space~$M$, in the sense that, for every lifted vector field~$V^w\eqdef \car\otimes w$, with~$w\in\mfX^\infty$, 
\begin{equation}\label{eq:IntrinsicVanishing}
\conn^\intr V^w \equiv 0\fstop
\end{equation}

\begin{remark}[Identification modulo~$\ker\div_\mu$]
For every pair of cylinder vector fields~$V,Z\in\VC{\infty}{b}$ and every cylinder function~$v\in\FC{\infty}{b}$ we have
\begin{equation}\label{eq:r:IntrConnectionModKerDiv}
(\conn^\intr_V Z)v = \tparen{\conn^\intr_{\Pi V} (\Pi Z)}v \fstop
\end{equation}
\begin{proof}
By~$\R$-bilinearity of~$\conn^\intr$ and $\R$-linearity of~$\Pi$, it suffices to verify the assertion when~$V,Z$ are elementary-tensor vector fields.
Letting~$Z=u\otimes w$ and applying~\eqref{eq:Connection},
\begin{align*}
\tparen{\conn^\intr_V (u\otimes w)}v &= \tparen{\conn^\intr (u\otimes w)}(V,\mbfD v)
\\
&= \iint \tparen{(\mbfD_\mu u) V_\mu}_x \cdot \tparen{(\mbfD_\mu v)w}_y \diff\mu^\otym{2}(x,y)
\\
&= \int \tparen{(\mbfD_\mu u) (\Pi_\mu V_\mu)}_x \diff\mu(x) \int \tparen{(\mbfD_\mu v) (\Pi_\mu w)}_y \diff\mu(y)
\\
&= \tparen{\conn^\intr (u\otimes \Pi_\emparg w)}(\Pi V,\mbfD v)
\\
&= \tparen{\conn^\intr_{\Pi V}(\Pi Z)} v \fstop \qedhere
\end{align*}
\end{proof}
\end{remark}

\subsubsection{Brackets}
Since the linear structure of~$\VC{\infty}{b}$ is inherited from that of~$\mfX^\infty$ by taking an $\FC{\infty}{b}$-linear extension, one is tempted to define a Lie bracket of vector fields 
\[
\pbracket{\emparg}{\emparg}\colon \VC{\infty}{b}\times \VC{\infty}{b} \to \VC{\infty}{b}
\]
as the~$\FC{\infty}{b}$-bilinear extension to~$\VC{\infty}{b}$ of the Lie bracket~$[\emparg,\emparg]$ on~$\mfX^\infty$.
It is readily verified that this definition is well-posed and gives rise to a Lie algebra structure on~$\VC{\infty}{b}$.
The symbol~$\pbracket{\emparg}{\emparg}$ is introduced for simplicity of notation: indeed, we might write~$\pbracket{\emparg}{\emparg}$ in a more suggestive way as the pullback to~$\VC{\infty}{b}$ of the bracket~$[\emparg,\emparg]$ on~$\mfX^\infty$ via the evaluation map~$\ev_\mu\colon \VC{\infty}{b}\to \mfX^\infty$, viz.
\[
\pbracket{V}{Z}_\mu = \tparen{{\ev_\mu}^* [\emparg,\emparg]}(V,Z) = [V_\mu,Z_\mu] \fstop
\]

However, this is \emph{not} the right definition in the spirit of differential geometry, since (by definition)
\begin{equation}\label{eq:PseudoBracket}
\pbracket{V}{uZ}= u \pbracket{V}{Z} \comma \qquad V,Z\in\VC{\infty}{b}\comma \quad u\in\FC{\infty}{b}\comma
\end{equation}
which means that~$\tparen{\VC{\infty}{b},\pbracket{\emparg}{\emparg}}$ is \emph{not} a Lie algebroid with identical anchor in the sense of e.g.~\cite[Dfn.~III.2.1, p.~100]{Mac87}.
Thus, we call~$\pbracket{\emparg}{\emparg}$ the \emph{pseudo-bracket} on~$\VC{\infty}{b}$.

Another natural definition of bracket is the following.

For cylinder vector fields~$V, Z\in\VC{\infty}{b}$ we set, for every~$u\in\FC{\infty}{b}$,
\begin{align}\label{eq:d:Bracket:0}
\tparen{\bbracket{V}{Z}u}_\mu \eqdef \tparen{V(Z u)}_\mu - \tparen{Z (Vu)}_\mu \fstop
\end{align}

\begin{proposition}[Bracket of vector fields]
The quantity~$\bbracket{V}{Z}$ defines a \emph{bracket} on cylinder vector fields~$V,Z\in \VC{\infty}{b}$.
\end{proposition}

\begin{proof}[Proof]
Let us first compute~$\bbracket{V}{Z}$ explicitly in order to show that it is a cylinder vector field.
By $\R$-linearity of~$V\mapsto Vu$ in~\eqref{eq:VFAction}, by $\R$-linearity of~$\conn^\intr$ in both variables, and by~$\R$-linearity of~$\bbracket{\emparg}{\emparg}$ in both variables, it suffices to show the assertion when~$V^1$,~$V^2$ are elementary tensors. Thus, let~$V^j=u^j\otimes w^j$, $j=1,2$, and~$u=F\circ \mbff^\trid$. Then,
\begin{align*}
\tparen{V^1 V^2 u}_\mu &= \mu\paren{\paren{\mbfD_\mu\tparen{(\mbfD_\emparg u)V^2_\emparg}^\trid}V^1_\mu}
\\
&= \mu\paren{\paren{\mbfD_\mu \paren{u^2 \sum_i^k (\partial_i F)\circ \mbff^\trid\cdot \tparen{(\diff f_i)w^2}^\trid }} V^1_\mu}
\\
&= \mu\paren{\sum_i^k (\partial_i F)(\mbff^\trid\mu) \cdot \tparen{(\diff f_i)w^2}^\trid\mu \cdot (\mbfD_\mu u^2)V^1_\mu}
\\
&\qquad + \mu\paren{u^2_\mu \sum_i \paren{\mbfD_\mu \tparen{(\partial_i F) \circ \mbff^\trid}}V^1_\mu \cdot \tparen{(\diff f_i)w^2}^\trid\mu}
\\
&\qquad + \mu\paren{u^2_\mu \sum_i (\partial_i F)(\mbff^\trid\mu) \cdot \tparen{\mbfD_\mu\tparen{(\diff f_i)w^2}^\trid} V^1_\mu}
\\
&= \mu \tparen{(\mbfD_\mu u)w^2} \cdot \mu\paren{(\mbfD_\mu u^2) V^1_\mu}
\\
&\qquad + u^2_\mu \sum_i^k\tparen{(\diff f_i)w^2}^\trid\mu \cdot \mu\paren{\paren{\mbfD_\mu \tparen{(\partial_i F) \circ \mbff^\trid}}V^1_\mu}
\\
&\qquad + u^2_\mu \sum_i^k (\partial_i F)(\mbff^\trid\mu) \cdot \mu\paren{\tparen{\mbfD_\mu\tparen{(\diff f_i)w^2}^\trid} V^1_\mu}
\\
&= \mu \tparen{(\mbfD_\mu u)w^2} \cdot \mu\paren{(\mbfD_\mu u^2) V^1_\mu}
\\
&\qquad + u^2_\mu \sum_i^k\tparen{(\diff f_i)w^2}^\trid\mu \cdot \mu\paren{\sum_j^k (\partial^2_{ij} F)(\mbff^\trid\mu) \cdot (\diff f_j) V^1_\mu}
\\
&\qquad + u^2_\mu \sum_i^k (\partial_i F)(\mbff^\trid\mu) \cdot \mu\paren{\diff\tparen{(\diff f_i)w^2} V^1_\mu}
\\
&= \mu \tparen{(\mbfD_\mu u)w^2} \cdot \mu\paren{(\mbfD_\mu u^2) V^1_\mu}
\\
&\qquad + u^2_\mu u^1_\mu \sum_{i,j} \tparen{(\diff f_i)w^2}^\trid\mu \cdot \tparen{(\diff f_j) w^1}^\trid\mu \cdot (\partial^2_{ij} F)(\mbff^\trid\mu)
\\
&\qquad + u^2_\mu u^1_\mu \sum_i (\partial_i F)(\mbff^\trid\mu) \cdot \mu\paren{w^1\tparen{(\diff f_i)w^2}} \fstop
\end{align*}

For a $1$-form~$\omega\in\Omega^\infty_1$ and vector fields~$w^1,w^2\in\mfX^\infty$ we have, by definition of the exterior differential on~$M$,
\[
\diff \omega (w^1,w^2) = w^1(\omega w^2) - w^2(\omega w^1) - \omega[w^1,w^2] \fstop
\]
Substituting~$\omega=\diff f_i$, we see that
\begin{equation}\label{eq:DifferentialBracket}
w^1\tparen{(\diff f_i) w^2} - w^2\tparen{(\diff f_i) w^1} = (\diff f_i) [w^1,w^2] \fstop
\end{equation}

Thus, we conclude that
\begin{align}
\nonumber
\tparen{\bbracket{V^1}{V^2}u}_\mu&\eqdef \tparen{V^1 V^2u}_\mu - \tparen{V^2V^1 u}_\mu
\\
\nonumber
&= \mu \tparen{(\mbfD_\mu u)w^2} \cdot \mu\paren{(\mbfD_\mu u^2) V^1_\mu} - \mu \tparen{(\mbfD_\mu u)w^1} \cdot \mu\paren{(\mbfD_\mu u^1) V^2_\mu}
\\
\nonumber
&\qquad + u^1_\mu u^2_\mu\cdot \mu \tparen{(\mbfD_\mu u)\bracket{w^1,w^2}}
\\
\label{eq:ExplicitBracket}
&= \mu \tparen{(\mbfD_\mu u)w^2} \cdot \mu\paren{(\mbfD_\mu u^2) V^1_\mu} - \mu \tparen{(\mbfD_\mu u)w^1} \cdot \mu\paren{(\mbfD_\mu u^1) V^2_\mu}
\\
\nonumber
&\qquad+\tparen{\pbracket{V^1}{V^2} u}_\mu \fstop
\end{align}
This shows that~$\bbracket{V^1}{V^2}$ acts as the derivation induced by a cylinder vector field in~$\VC{\infty}{b}$, and thus that it can be identified with an element of~$\VC{\infty}{b}$.

The $\R$-bilinearity follows from~\eqref{eq:VFAction}. The alternating property, anticommutativity, and the Jacobi identity hold since~$\bbracket{\emparg}{\emparg}$ is a commutator by definition.
\end{proof}

We are now in the position to easily check that the bracket~$\bbracket{\emparg}{\emparg}$ is the right one according to our differential-geometric understanding.

\begin{corollary}\label{c:LieAlgebroid}
The space~$\tparen{\VC{\infty}{b},\bbracket{\emparg}{\emparg}}$ is a Lie algebroid on~$\FC{\infty}{b}$ with identical anchor, that is, for every~$V, Z\in\VC{\infty}{b}$ and every~$v\in\FC{\infty}{b}$,
\begin{equation}\label{eq:c:LieAlgebroid:0}
\bbracket{V}{vZ} = v\bbracket{V}{Z}+ (Vv) Z \fstop
\end{equation}

\begin{proof}
Since~$\bbracket{\emparg}{\emparg}$ is a commutator,~$\tparen{\VC{\infty}{b},\bbracket{\emparg}{\emparg}}$ is a Lie algebra.
As for~\eqref{eq:c:LieAlgebroid:0}, by~\eqref{eq:p:Connection:0} we have
\begin{align*}
&\tparen{(\conn^\intr _\mu v Z)(V,\mbfD_\emparg u)}_\mu - \tparen{(\conn^\intr _\mu V)(v Z,\mbfD_\emparg u)}_\mu
\\
&= v_\mu\cdot \tparen{(\conn^\intr _\mu Z)(V,\mbfD_\emparg u)}_\mu + (Vv)_\mu \cdot \tparen{(\mbfD_\mu u) Z}_\mu - v_\mu \cdot \tparen{(\conn^\intr _\mu V)(Z,\mbfD_\emparg u)}_\mu
\\
&= v_\mu \paren{ \tparen{(\conn^\intr _\mu Z)(V,\mbfD_\emparg u)}_\mu - \tparen{(\conn^\intr _\mu V) (Z,\mbfD_\emparg u)}_\mu } + (Vv)_\mu \cdot (Zu)_\mu \fstop
\end{align*}
and the conclusion follows by combining the above equality with~\eqref{eq:PseudoBracket}.
\end{proof}
\end{corollary}

\begin{remark}[Constant fields]\label{r:BracketConstantFields}
The definition of the bracket is the first instance of the fact that it is not sufficient, in order to tackle the differential geometry of~$\msP_2$-spaces, to work with constant fields.
Indeed, for constant cylinder vector fields~$V^j\eqdef\car\otimes w^j$, with~$w^j\in\mfX^\infty$ for~$j=1,2$,
\[
\bbracket{V^{w_1}}{V^{w_2}} = \pbracket{V^{w_1}}{V^{w_2}} \comma
\]
that is, on constant fields the bracket reduces to the pseudo-bracket.
\end{remark}

\begin{remark}[Comparison with~\cite{Lot07}]
Our definition of bracket extends the one of~\cite{Lot07}, i.e.\ the former coincides with the latter when restricted to constant gradient-type vector fields and to measures~$\mu\in\msP^\infty$.
\begin{proof}
Let~$F=f^\trid$ and~$\mu=\rho\vol_g$.
For every~$\phi_1,\phi_2\in\Cinfty$, and setting for simplicity of notation~$w^k\eqdef \nabla \phi_k$ for~$k=1,2$, we have, in the notation of~\cite{Lot07},
\begin{align}
\tag*{\cite[(4.2)]{Lot07}}
&\tparen{[V_{\phi_1},V_{\phi_2}] F}(\rho\, \vol_g) 
\\
\nonumber
&\qquad = \diff_\eps\restr_{\eps=0} F\paren{\rho\, \vol_g + \eps \nabla^j\tparen{\nabla^i(\rho\, \nabla_i \phi_1) \nabla_j\phi_2} -\eps \nabla^j\tparen{\nabla^i(\rho\, \nabla_i \phi_2) \nabla_j\phi_1} }
\\
\nonumber
&\qquad = \int f \paren{\nabla^j\tparen{\nabla^i(\rho\, \nabla_i \phi_1) \nabla_j\phi_2} - \nabla^j\tparen{\nabla^i(\rho\, \nabla_i \phi_2) \nabla_j\phi_1} }\dvol_g
\\
\nonumber
&\qquad = \int f \paren{\div_g\tparen{\div_g (\rho w^1) w^2}- \div_g\tparen{\div_g (\rho w^2) w^1}} \dvol_g
\\
\nonumber
&\qquad = - \int \diff f \paren{\div_g(\rho w^1) w^2 - \div_g(\rho w^2)w^1} \dvol_g
\\
\nonumber
&\qquad = - \int \diff f \tparen{(\diff\rho\, w^1) w^2 + \rho (\div_g w^1) w^2 - (\diff\rho\, w^2) w^1 - \rho (\div_g w^2) w^1} \dvol_g
\\
\nonumber
&\qquad = - \int \tparen{(\diff\rho\, w^1) (\diff f\, w^2) - (\diff\rho\, w^2) (\diff f\, w^1) + \rho (\div_g w^1) (\diff f\, w^2) - \rho (\div_g w^2) (\diff f\, w^1)} \dvol_g
\\
\nonumber
&\qquad = - \int \diff\rho \tparen{(\diff f\, w^2)w^1 - (\diff f\, w^1)w^2} \dvol_g - \int \tparen{(\div_g w^1) (\diff f\, w^2) - (\div_g w^2) (\diff f\, w^1)} \rho\, \dvol_g
\\
\nonumber
&\qquad = \int \tparen{\diff\tparen{\diff f\, w^2} w^1 - \diff\tparen{\diff f\, w^1} w^2} \rho\, \dvol_g
\\
\nonumber
&\qquad\qquad+ \int \tparen{(\diff f\, w^2)(\div_g w^1) - (\diff f\, w^1)(\div_g w^2)}\rho\, \dvol_g
\\
\nonumber
&\qquad\qquad - \int \tparen{(\div_g w^1) (\diff f\, w^2) - (\div_g w^2) (\diff f\, w^1)} \rho\, \dvol_g
\\
\nonumber
&\qquad =  \int \tparen{w^1(\diff f\, w^2) - w^2(\diff f\, w^1)} \rho\, \dvol_g
\\
\label{eq:BracketLott}
&\qquad =  \int \tparen{\diff f\,  [w^1,w^2]} \rho\, \dvol_g \fstop
\end{align}
by~\eqref{eq:DifferentialBracket}.
By the chain rule,~\eqref{eq:BracketLott} extends to all~$F\in\FC{\infty}{b}$ and we thus have
\[
\tparen{[V_{\phi_1},V_{\phi_2}] F}(\rho\, \vol_g) = \pbracket{V_{\phi_1}}{V_{\phi_2}}_{\rho\vol_g} \fstop \qedhere
\]
\end{proof}
\end{remark}

\subsubsection{Torsion}
We proceed to compute the \emph{torsion} of~$\conn^\intr$, i.e.\ the operator
\[
\mbfT^\intr\colon {\VC{\infty}{b}}^\tym{2}\to \VC{\infty}{b}
\]
satisfying, for every~$(V,Z)\in {\VC{\infty}{b}}^\tym{2}$ and every~$u\in\FC{\infty}{b}$,
\[
\tparen{\mbfT^\intr(V,Z)u}_\mu\eqdef \tparen{(\conn^\intr_\mu Z)(V,\mbfD_\emparg u)}_\mu - \tparen{(\conn^\intr_\mu V)(Z,\mbfD_\emparg u)}_\mu - \tparen{\bbracket{V}{Z}u}_\mu \fstop
\]

\begin{proposition}[Torsion]\label{p:Torsion}
The torsion~$\mbfT^\intr$ 
satisfies, for every~$(V,Z)\in {\VC{\infty}{b}}^\tym{2}$,
\begin{equation}\label{eq:p:Torsion:0}
\mbfT^\intr(V,Z)= -\pbracket{V}{Z} \fstop
\end{equation}

\begin{proof}
Immediate consequence of the computation of the bracket in~\eqref{eq:ExplicitBracket}.
\end{proof}
\end{proposition}

\subsubsection{Levi-Civita connection}
Our next objective is to define the Levi-Civita connection on~$\msP$.
We start by recalling that an affine connection~$\nabla$ on~$M$ is called \emph{metric-compatible} if~$\nabla g=0$, that is, for every~$w^1,w^2,w^3\in \mfX^\infty$,
\begin{equation}\label{eq:MetricCompatibleM}
w^1 g(w^2,w^3)= g(\nabla_{w^1}w^2,w^3)+g(w^2,\nabla_{w^1}w^3) \fstop
\end{equation}
The Fundamental Theorem of Riemannian geometry states that there exists a unique metric-compatible torsion-free connection on~$(M,g)$, called the \emph{Levi-Civita connection} and denoted by~$\nabla^\lc$.

\begin{definition}[Connection pullback]
Let~$\nabla$ be any affine connection on~$M$.
We define its \emph{pullback} via the evaluation map as the $\FC{\infty}{b}$-linear extension~$\boldOmega^\nabla \colon \VC{\infty}{b}\to \OC{1}{\infty}{b}\otimes \VC{\infty}{b}$ of~$\nabla\colon \mfX^\infty \to \Omega^\infty_1\otimes \mfX^\infty$, satisfying
\[
\tparen{(\boldOmega^\nabla Z)(V,\Omega)}_\mu \eqdef \mu\tparen{\Omega_\mu (\nabla_{V_\mu} Z_\mu)} \fstop
\]
\end{definition}

Note that, since~$\OC{1}{\infty}{b}\otimes\VC{\infty}{b}$ and~$\End(\VC{\infty}{b})$ are canonically isomorphic, we may equivalently regard~$\boldOmega^\nabla$ as a $\End(\VC{\infty}{b})$-valued $1$-form~$\boldOmega\colon \VC{\infty}{b} \to \End(\VC{\infty}{b})$.

We call~$\boldOmega^\nabla$ the \emph{connection pullback}, since for every pair of lifted vector fields $V^j\eqdef \car\otimes w^j$, with~$w^j\in\mfX^\infty$ for~$j=1,2$, we have
\begin{equation}\label{eq:ConnectionPullback}
(\boldOmega^\nabla V^1)V^2 = \car\otimes \nabla_{w^2} w^1 \fstop
\end{equation}

\begin{remark}
Note that~$\boldOmega^\nabla$ is \emph{not} an affine connection on~$\msP$. Indeed, since~$v_\mu$ is a scalar on~$M$,
\begin{equation}\label{eq:OmegaConnTensorial}
\tparen{\boldOmega^\nabla(v Z) (V,\Omega)}_\mu = \mu\tparen{\Omega_\mu\nabla_{V_\mu}(v_\mu Z_\mu)} = v_\mu \,  \mu\tparen{\Omega_\mu(\nabla_{V_\mu}Z_\mu)}\fstop
\end{equation}
Thus~$\boldOmega^\nabla$ is tensorial, and does not verify the product rule~\eqref{eq:p:Connection:0} (with~$\boldOmega^\nabla$ in place of~$\conn^\intr$).
\end{remark}

\paragraph{Metric on~$\msP$ and metric-compatibility}
A Riemannian metric~$g$ on~$M$ lifts to a `Riemannian metric'~$\mbfG^\der$ on~$T^\der\msP$ by setting
\begin{equation}\label{eq:Metric}
\mbfG^\der_\mu(V,Z) = \tparen{\mbfG^\der(V,Z)}_\mu \eqdef \mu\tparen{g(V_\mu,Z_\mu)} \comma
\end{equation}
and to a `semi-definite Riemannian metric'~$\mbfG^\grad$ on~$T^\der\msP$ by setting
\begin{equation}\label{eq:MetricNabla}
\mbfG^\grad_\mu(V,Z) = \tparen{\mbfG^\der(\Pi V, \Pi Z)}_\mu \eqdef \mu\tparen{g(\Pi_\mu V_\mu,\Pi_\mu Z_\mu)} \fstop
\end{equation}

The non-relabeled restriction of~$\mbfG^\grad$ to~$T^\grad\msP$ is a `Riemannian metric' on~$T^\grad\msP$, namely Otto's formal Riemannian metric~\cite{Ott01}, and we have
\[
\mbfG^\grad \equiv \mbfG^\der \quad \text{on} \quad \VC{\infty}{b,\grad} \fstop
\]

By analogy with the case of~$(M,g)$, we say that an affine connection~$\conn$ on~$\msP$ is \emph{$\mbfG^\der$-metric-compatible} if and only if, for every~$V^1,V^2,V^3\in\VC{\infty}{b}$,
\begin{equation}\label{eq:MetricCompatibleP}
V^1\mbfG^\der(V^2,V^3) = \mbfG^\der(\conn_{V^1} V^2,V^3) + \mbfG^\der(V^2, \conn_{V^1} V^3) \comma
\end{equation}
and analogously that~$\conn$ is \emph{$\mbfG^\grad$-metric-compatible} if~\eqref{eq:MetricCompatibleP} holds with~$\mbfG^\grad$ in place of~$\mbfG^\der$ for every~$V^1,V^2,V^3\in\VC{\infty}{b,\grad}$.

In the next result, let~$\boldOmega^\lc\eqdef \boldOmega^{\nabla^\lc}$, and, for any affine connection~$\nabla$ on~$M$, set
\begin{equation}\label{eq:TrueConnection}
\conn^\nabla\eqdef \conn^\intr + \boldOmega^\nabla \fstop
\end{equation}
In light of the product rule for~$\conn^\intr$ in Proposition~\ref{p:Connection} and since~$\boldOmega^\nabla$ is a $\End\tparen{\VC{\infty}{b}}$-valued $1$-form (in particular, since~$\boldOmega^\nabla$ is tensorial, i.e.\ satisfying~\eqref{eq:OmegaConnTensorial}), the operator~$\conn^\nabla$ also satisfies the product rule~\eqref{eq:p:Connection:0} and is therefore a connection.

\begin{theorem}[Levi-Civita connection]\label{t:LCConnection}
For each connection~$\nabla$ on~$(M,g)$,
\begin{enumerate}[$(i)$]
\item\label{i:t:LCConnection:1} $\conn^\nabla$ is torsion-free if and only if $\nabla$ is torsion-free;
\item\label{i:t:LCConnection:2} $\conn^\nabla$ is $\mbfG^\der$-metric-compatible if and only if $\nabla$ is $g$-metric-compatible;
\item\label{i:t:LCConnection:3} $\conn^\nabla$ is $\mbfG^\der$-metric-compatible and torsion-free if and only if $\nabla=\nabla^\lc$ is the Levi-Civita connection on~$(M,g)$;
\end{enumerate}
and $\conn^\lc\eqdef\conn^{\nabla^\lc}$ is the unique $\mbfG^\der$-metric-compatible torsion-free affine connection on~$\msP$.
\end{theorem}

\begin{proof}
By $\R$-multilinearity of all the tensors involved, it suffices to show the statement for elementary-tensor cylinder vector fields.
Thus, everywhere in the following let~$V^j\eqdef u^j\otimes w^j\in\VC{\infty}{b}$ for~$j=1,2,3$ be arbitrary if not otherwise specified.

\paragraph{Proof of~\ref{i:t:LCConnection:1}}
Assume~$T^\nabla=0$.
By definition of~$\conn^\nabla$,~\eqref{eq:p:Torsion:0} and~\eqref{eq:ConnectionPullback}, for every~$V,Z\in\VC{\infty}{b}$ and every elementary-tensor $1$-form~$\Omega= u\otimes\omega \in \OC{1}{\infty}{b}$,
\begin{align*}
\tparen{\Omega\,\mbfT^\nabla(V^1,V^2)}_\mu &= \tparen{\Omega\mbfT^\intr(V^1,V^2) + (\boldOmega^\nabla V^2)(V^1,\Omega) - (\boldOmega^\nabla V^1)(V^2,\Omega)}_\mu 
\\
&= \mu\tparen{-\Omega_\mu\pbracket{V^1_\mu}{V^2_\mu}+ \Omega_\mu(\nabla_{V^1_\mu}V^2_\mu)- \Omega_\mu(\nabla_{V^2_\mu}V^1_\mu)}
\\
&= (u^1u^2 u)_\mu \cdot \mu \tparen{-\omega[w^1,w^2]+\omega(\nabla_{w^1}w^2)-\omega(\nabla_{w^2}w^1)}
\\
&= (u^1u^2 u)_\mu \cdot \mu \paren{\omega \tparen{T^\nabla(w^1,w^2)}}
\\
&=0
\end{align*}
by the assumption.

\medskip

Vice versa, assume~$\mbfT^\nabla=0$.
Let~$w^j\in\mfX^\infty$ with~$j=1,2$ and set~$V^j\eqdef \car\otimes w^j$.
By~\eqref{eq:IntrinsicVanishing}, \eqref{eq:ConnectionPullback}, and~\eqref{eq:p:Torsion:0} we have
\[
\mbfT^\nabla(V^1,V^2)= (\boldOmega^\nabla V^2)(V^1) - (\boldOmega^\nabla V^1)(V^2) - \pbracket{V^1}{V^2} = T^\nabla(w^1,w^2) \comma
\]
and the conclusion follows by the assumption.

\paragraph{Proof of~\ref{i:t:LCConnection:2}}
Assume~$\nabla g=0$. Unraveling the definitions of the objects in~\eqref{eq:MetricCompatibleP}, we need to show
\begin{equation}\label{eq:t:LCConnection:1}
\paren{\paren{\mbfD \tparen{\mbfG^\der_\emparg(V^2_\emparg,V^3_\emparg)} }V^1_\emparg}_\mu \overset{?}= \mu\tparen{g \tparen{(\conn_{V^1} V^2)_\mu, V^3_\mu}} + \mu\tparen{g\tparen{V^2_\mu, (\conn_{V^1} V^3)_\mu}} \fstop
\end{equation}
For the left-hand side of~\eqref{eq:t:LCConnection:1},
\begin{align*}
&\paren{\paren{\mbfD_\mu \tparen{\mbfG^\der_\emparg(V^2,V^3)} }V^1}_\mu 
\\
& \qquad = \mu\paren{ \mbfD_\mu\tparen{g(V^2_\emparg,V^3_\emparg)^\trid} V^1_\mu}
\\
& \qquad = \mu\paren{\mbfD_\mu\tparen{u^2 u^3\, g(w^2,w^3)^\trid} V^1_\mu}
\\
& \qquad = u^3_\mu \cdot \mu\tparen{g(w^2,w^3)} \cdot \mu\tparen{(\mbfD_\mu u^2) V^1_\mu} + u^2_\mu \cdot \mu\tparen{g(w^2,w^3)} \cdot \mu\tparen{(\mbfD_\mu u^3) V^1_\mu}
\\
&\qquad\qquad + u^1_\mu u^2_\mu u^3_\mu \cdot \mu\paren{ \diff \tparen{g(w^2,w^3)} w^1}
\\
& \qquad = u^3_\mu \cdot \mu\tparen{g(w^2,w^3)} \cdot \mu\tparen{(\mbfD_\mu u^2) V^1_\mu} + u^2_\mu \cdot \mu\tparen{g(w^2,w^3)} \cdot \mu\tparen{(\mbfD_\mu u^3) V^1_\mu}
\\
&\qquad\qquad + (u^1 u^2 u^3)_\mu \cdot \mu\paren{w^1 g(w^2,w^3)}\comma
\end{align*}
whence, by the assumption in the form~\eqref{eq:MetricCompatibleM}, we conclude
\begin{equation}\label{eq:t:LCConnection:2}
\begin{aligned}
&\paren{\paren{\mbfD_\mu \tparen{\mbfG^\der_\emparg(V^2,V^3)} }V^1}_\mu 
\\
& \qquad = u^3_\mu \cdot \mu\tparen{g(w^2,w^3)} \cdot \mu\tparen{(\mbfD_\mu u^2) V^1_\mu} + u^2_\mu \cdot \mu\tparen{g(w^2,w^3)} \cdot \mu\tparen{(\mbfD_\mu u^3) V^1_\mu}
\\
&\qquad\qquad + (u^1 u^2 u^3)_\mu \cdot \mu\paren{g(\nabla_{w^1}w^2,w^3)} + (u^1 u^2 u^3)_\mu \cdot \mu\paren{g(w^2,\nabla_{w^1} w^3)}\fstop
\end{aligned}
\end{equation}

For the right-hand side of~\eqref{eq:t:LCConnection:1},
\begin{align*}
\mu \paren{g \tparen{(\conn_{V^1} V^2)_\mu, V^3_\mu}} &= u^3_\mu\cdot \mu\tparen{(\mbfD_\mu u^2) V^1_\mu} \cdot \mu\paren{g\paren{w^2, w^3}} 
+ u^3_\mu (u^1u^2)_\mu \cdot \mu\tparen{g(\nabla_{w^1}w^2,w^3)} \comma
\\
\mu \paren{g \tparen{(\conn_{V^1} V^3)_\mu, V^2_\mu}} &= u^2_\mu\cdot \mu\tparen{(\mbfD_\mu u^3) V^1_\mu} \cdot \mu\paren{g\paren{w^3, w^2}}
+u^2_\mu (u^1u^3)_\mu \cdot \mu\tparen{g(w^2,\nabla_{w^1}w^3)} \comma
\end{align*}
and the conclusion follows by comparison with~\eqref{eq:t:LCConnection:2}.

\medskip

Vice versa, assume~$\conn^\nabla\mbfG^\der=0$, i.e.~\eqref{eq:MetricCompatibleP} holds.
Let~$w^j\in\mfX^\infty$ with~$j=1,2$ and set~$V^j\eqdef \car\otimes w^j$.
Applying~\eqref{eq:IntrinsicVanishing},~\eqref{eq:ConnectionPullback}, and~\eqref{eq:Metric} to~\eqref{eq:MetricCompatibleP} we have, for every~$\mu\in\msP$,
\[
\mu\tparen{V^1 g(w^2,w^3)^\trid} = \mu\tparen{g (\nabla_{w^1}w^2, w^3) } +  \mu\tparen{g (w^2, \nabla_{w^1}w^3) }\fstop
\]
Since
\[
\mu\tparen{V^1 g(w^2,w^3)^\trid} = \mu\paren{\mbfD \tparen{g(w^2,w^3)^\trid} V^1} = \mu\tparen{\diff \tparen{g(w^2,w^3)} w^1} = \mu\tparen{w^1 g(w^2,w^3) } \comma
\]
we have
\[
0= \mu\tparen{w^1 g(w^2,w^3) - g (\nabla_{w^1}w^2, w^3) - g (w^2, \nabla_{w^1}w^3) } \comma
\]
and the conclusion follows since~$\mu\in\msP$ is arbitrary.

\paragraph{Proof of~\ref{i:t:LCConnection:3}}
Follows by combining~\ref{i:t:LCConnection:1} and~\ref{i:t:LCConnection:2}.

\paragraph{Proof of uniqueness}
The proof of uniqueness of the Levi-Civita connection on a finite-dimensional Riemannian manifold in~\cite[Thm.~5.10, p.~122]{Lee18} is purely algebraic, thus it also applies to our setting \emph{mutatis mutandis} in light of the non-degeneracy of~$\mbfG^\der$ on~$\VC{\infty}{b}$.
\end{proof}

\begin{remark}[Comparison with~\cite{Lot07}]\label{r:RemarkLott2}
Our definition of the Levi-Civita connection extends the one of~\cite{Lot07}, i.e.\ the former coincides with the latter when restricted to constant gradient-type vector fields and to measures~$\mu\in\msP^\infty$.
\begin{proof}
Let~$F=f^\trid$ and~$\mu=\rho\vol_g$.
For every~$\phi_1,\phi_2\in\Cinfty$, and setting for notational simplicity~$w^k\eqdef \nabla \phi_k$ for~$k=1,2$, we have, in the notation of~\cite{Lot07},
\begin{align*}
\tag*{\cite[(4.4)]{Lot07}}
&\tparen{\overline\nabla_{V_{\phi_1}} V_{\phi_2} F}_{\rho\vol_g}
\\
\nonumber
&\qquad = \diff_\eps\restr_{\eps=0} F\paren{\rho\,\vol_g - \eps \nabla_i \paren{\rho \nabla_j \phi_1\, \nabla^i\nabla^j\phi_2} \vol_g}
\\
\nonumber
&\qquad= -\int f\, \nabla_i \paren{\rho \nabla_j\phi_1\, \nabla^i\nabla^j\phi_2}\dvol_g
\\
\nonumber
&\qquad= \int \nabla_i f\, \nabla_j\phi_1\, \nabla^i\nabla^j \phi_2\, \rho\, \dvol_g
\\
\nonumber
&\qquad= \int \Hess_g(\phi_2)(\nabla f, w^1)\, \rho\, \dvol_g
\\
&\qquad= \int \paren{w^1\tparen{\nabla f (\phi_2)} - \tparen{\nabla^\lc_{w^1} \nabla f}(\phi_2)} \rho\, \dvol_g
\\
&\qquad= \int \paren{w^1g\tparen{\nabla f, w^2} - g\tparen{\nabla\phi_2, \nabla^\lc_{w^1} \nabla f}} \rho\, \dvol_g
\\
&\qquad= \int \paren{g\tparen{\nabla^\lc_{w^1}\nabla f, w^2} + g\tparen{\nabla f, \nabla^\lc_{w^1}w^2} - g\tparen{w^2, \nabla^\lc_{w^1} \nabla f}} \rho\, \dvol_g
\\
&\qquad= \int \paren{(\nabla^\lc_{w^1}w^2) f} \rho\, \dvol_g\comma
\end{align*}
where we used~\eqref{eq:MetricCompatibleM}.
By the chain rule, the above computation extends to every~$F\in\FC{\infty}{b}$.
Thus, Lott's connection~$\overline\nabla$ satisfies, for every~$\phi_1,\phi_2\in\Cinfty$ and~$F\in \FC{\infty}{b}$
\begin{align}
\tparen{\overline\nabla_{V_{\phi_1}} V_{\phi_2} F}_{\rho\vol_g} = \tparen{\tparen{(\boldOmega^\nabla V_{\phi_2}) V_{\phi_1}} F}_{\rho\vol_g} 
\fstop
\end{align}
We conclude that
\begin{equation}\label{r:LottConnection}
\overline\nabla_{V_{\phi_1}} V_{\phi_2} = (\boldOmega^\nabla V_{\phi_2}) V_{\phi_1}\comma
\end{equation}
and the assertion follows from~\eqref{eq:IntrinsicVanishing}.
\end{proof}
\end{remark}

We note that~\eqref{r:LottConnection} greatly simplifies the expression for the connection on constant fields given in~\cite[Lem.~4]{Lot07}, especially in that it makes the resulting constant field \emph{explicitly} independent from the application point~$\rho\vol_g$.
This does not contradict the fact that $\Pi\VC{\infty}{b}$-valued Lie brackets of constant vector fields are non-constant~\cite[p.~3, after Eqn.~(7)]{DinFanLi25}, since the equality above is always interpreted for vector fields acting on cylinder functions, and therefore modulo~$\ker\div_\mu$.

\subsubsection{Riemann tensor}
We proceed to compute the Riemann tensor
\begin{equation}\label{eq:RiemannTensor}
\mbfR^\der(V^1,V^2)V^3 \eqdef \conn^\lc_{V^1} \conn^\lc_{V^2} V^3 -  \conn^\lc_{V^2} \conn^\lc_{V^1} V^3 - \conn^\lc_{\bbracket{V^1}{V^2}}V^3 \fstop
\end{equation}

In light of the product rule for~$\conn^\lc$ (cfr.\ Proposition~\ref{p:Connection}) and of the anchor property for the Lie bracket of vector fields (see Corollary~\ref{c:LieAlgebroid}), that the expression in~\eqref{eq:RiemannTensor} is indeed tensorial follows from the algebraic proof for the standard Riemann tensor, e.g.~\cite[Prop.~7.3, p.~196]{Lee18}.

\begin{theorem}\label{t:RiemannTensor}
The Riemann tensor~\eqref{eq:RiemannTensor} satisfies
\begin{equation}\label{eq:t:RiemannTensor:0}
\mbfR^\der(V^1,V^2)V^3 = R(V^1_\emparg,V^2_\emparg)V^3_\emparg \fstop
\end{equation}

\begin{proof}
We have
\begin{align}
\mbfR^\der&(V^1,V^2)V^3
\\
\label{eq:t:Riemann:1}
&=\conn^\intr_{V^1} \conn^\intr_{V^2} V^3 - \conn^\intr_{V^2} \conn^\intr_{V^1} V^3 - \conn^\intr_{\bbracket{V^1}{V^2}}V^3
\\
\label{eq:t:Riemann:2}
&\qquad\begin{aligned}&+\tparen{\boldOmega^\lc\tparen{(\boldOmega^\lc V^3) V^2}}V^1 - \tparen{\boldOmega^\lc \tparen{(\boldOmega^\lc V^3) V^1}}V^2
\\
&\qquad - \tparen{\boldOmega^\lc V^3}\bbracket{V^1}{V^2}\end{aligned}
\\
\label{eq:t:Riemann:3}
&\qquad+ \conn^\intr_{V^1}\tparen{(\boldOmega^\lc V^3) V^2} - \conn^\intr_{V^2}\tparen{(\boldOmega^\lc V^3) V^1}
\\
\label{eq:t:Riemann:4}
&\qquad+ \tparen{\boldOmega^\lc \tparen{\conn^\intr_{V^2} V^3}}V^1 - \tparen{\boldOmega^\lc \tparen{\conn^\intr_{V^1} V^3}}V^2\fstop
\end{align}

By $\R$-multilinearity of all the operators involved, it suffices to compute the Riemann tensor on arbitrary elementary-tensor cylinder vector fields $V^j\eqdef u^j\otimes w^j\in\VC{\infty}{b}$, with~$j=1,2,3$.

\paragraph{Intrinsic terms}
Letting~$u^3= F^3\circ\mbff^{3\trid} \in\FC{\infty}{b}$,
\begin{align*}
&\conn^\intr_{V^1} \conn^\intr_{V^2} V^3 
\\
&\qquad = \paren{\mbfD \paren{ \tparen{(\mbfD_\emparg u^3) V^2_\emparg}^\trid} V^1_\emparg} \otimes w^3
\\
&\qquad =\tparen{(\mbfD u^2) V^1}^\trid \cdot \tparen{(\mbfD u^3) w^2}^\trid \otimes w^3
\\
&\qquad\qquad + \sum_{i,k} (\partial^2_{ik} F^3) \circ \mbff^{3\trid} \cdot \tparen{(\diff f_k)V^1_\emparg}^\trid \cdot \tparen{(\diff f_i)V^2_\emparg}^\trid \otimes w^3
\\
&\qquad\qquad +\sum_i (\partial_i F^3)\circ\mbff^{3\trid} \cdot \paren{\paren{\diff\tparen{(\diff f^3_i) V^2_\emparg}} (V^1_\emparg)}^\trid \otimes w^3
\\
&\qquad = \paren{\tparen{(\mbfD u^2) V^1}^\trid \cdot \tparen{(\mbfD u^3) w^2}^\trid + (\mbfD^\sym{2} u^3)(V^1_\emparg,V^2_\emparg) + \tparen{V^1_\emparg \tparen{(\mbfD_\emparg u^3)V^2_\emparg}}^\trid} \otimes w^3 \fstop
\end{align*}

Thus, on the one hand,
\begin{align*}
\nonumber
&\conn^\intr_{V^1} \conn^\intr_{V^2} V^3 -  \conn^\intr_{V^2} \conn^\intr_{V^1} V^3
\\
&\qquad = \paren{\tparen{(\mbfD u^2) V^1}^\trid \cdot \tparen{(\mbfD u^3) w^2}^\trid - \tparen{(\mbfD u^1) V^2}^\trid \cdot \tparen{(\mbfD u^3) w^1}^\trid}\otimes w^3
\\
&\qquad\qquad + \paren{\tparen{V^1_\emparg \tparen{(\mbfD_\emparg u^3)V^2_\emparg}}^\trid - \tparen{V^2_\emparg \tparen{(\mbfD_\emparg u^3)V^1_\emparg}}^\trid} \otimes w^3
\\
&\qquad = \paren{\tparen{(\mbfD u^2) V^1}^\trid \cdot \tparen{(\mbfD u^3) w^2}^\trid - \tparen{(\mbfD u^1) V^2}^\trid \cdot \tparen{(\mbfD u^3) w^1}^\trid}\otimes w^3
\\
&\qquad\qquad+ u^1u^2\paren{\tparen{w^1 \tparen{(\mbfD_\emparg u^3)w^2}}^\trid - \tparen{w^2 \tparen{(\mbfD_\emparg u^3)w^1}}^\trid} \otimes w^3
\\
&\qquad = \paren{\tparen{(\mbfD u^2) V^1}^\trid \cdot \tparen{(\mbfD u^3) w^2}^\trid - \tparen{(\mbfD u^1) V^2}^\trid \cdot \tparen{(\mbfD u^3) w^1}^\trid}\otimes w^3
\\
&\qquad\qquad+ u^1u^2\paren{\paren{w^1 \tparen{(\mbfD_\emparg u^3)w^2} - w^2 \tparen{(\mbfD_\emparg u^3)w^1}}^\trid} \otimes w^3
\\
&\qquad = \paren{\tparen{(\mbfD u^2) V^1}^\trid \cdot \tparen{(\mbfD u^3) w^2}^\trid - \tparen{(\mbfD u^1) V^2}^\trid \cdot \tparen{(\mbfD u^3) w^1}^\trid}\otimes w^3
\\
&\qquad\qquad+ u^1u^2\paren{\paren{(\mbfD_\emparg u^3)[w^1,w^2]}^\trid} \otimes w^3
\\
&\qquad = \paren{\tparen{(\mbfD u^2) V^1}^\trid \cdot \tparen{(\mbfD u^3) w^2}^\trid - \tparen{(\mbfD u^1) V^2}^\trid \cdot \tparen{(\mbfD u^3) w^1}^\trid}\otimes w^3
\\
&\qquad\qquad+ {\paren{(\mbfD_\emparg u^3)\pbracket{V^1}{V^2}_\emparg}^\trid} \otimes w^3
\\
&\qquad = \paren{\tparen{(\mbfD u^2) V^1}^\trid \cdot \tparen{(\mbfD u^3) w^2}^\trid - \tparen{(\mbfD u^1) V^2}^\trid \cdot \tparen{(\mbfD u^3) w^1}^\trid}\otimes w^3
\\
&\qquad\qquad+ \tparen{\pbracket{V^1}{V^2} u^3} \otimes w^3
\\
&\qquad = \tparen{\bbracket{V^1}{V^2}u^3} \otimes w^3
\end{align*}
by~\eqref{eq:DifferentialBracket} and~\eqref{eq:ExplicitBracket}.
On the other hand,
\begin{align*}
\conn^\intr_{\bbracket{V^1}{V^2}}V^3 &= \tparen{(\mbfD_\emparg u^3) \bbracket{V^1}{V^2}_\emparg}^\trid \otimes w^3 = \tparen{\bbracket{V^1}{V^2}u^3} \otimes w^3\fstop
\end{align*}
Comparing the two expressions we see that~\eqref{eq:t:Riemann:1} vanishes.

\paragraph{Pullback terms}
As for the terms in~\eqref{eq:t:Riemann:2}, by~\eqref{eq:ExplicitBracket} we have
\begin{align}
\nonumber
&\hspace{-2em}\tparen{\boldOmega^\lc\tparen{(\boldOmega^\lc V^3) V^2}}V^1 - \tparen{\boldOmega^\lc \tparen{(\boldOmega^\lc V^3) V^1}}V^2
- \tparen{\boldOmega^\lc V^3}\bbracket{V^1}{V^2}
\\
\nonumber
&= \nabla^\lc_{V^1_\emparg} \nabla^\lc_{V^2_\emparg} V^3_\emparg - \nabla^\lc_{V^2_\emparg} \nabla^\lc_{V^1_\emparg} V^3_\emparg - \nabla^\lc_{\bbracket{V_1}{V_2}_\emparg} V^3_\emparg
\\
\nonumber
&= u^1u^2u^3 \otimes \nabla^\lc_{w^1}\nabla^\lc_{w^2}w^3 - u^1u^2u^3 \otimes \nabla^\lc_{w^2}\nabla^\lc_{w^1}w^3 
\\
\nonumber
&\qquad - \nabla^\lc_{(\conn^\intr_{V^1}V^2)_\emparg-(\conn^\intr_{V^2}V^1)_\emparg + \pbracket{V^1}{V^2}_\emparg} V^3_\emparg
\\
\nonumber
&= u^1u^2u^3 \otimes \tparen{\nabla^\lc_{w^1}\nabla^\lc_{w^2}w^3 - \nabla^\lc_{w^2}\nabla^\lc_{w^1}w^3 }
\\
\nonumber
&\qquad - \nabla^\lc_{(\conn^\intr_{V^1}V^2)_\emparg} V^3 + \nabla^\lc_{(\conn^\intr_{V^2}V^1)_\emparg} V^3 -\nabla^\lc_{\pbracket{V^1}{V^2}_\emparg} V^3_\emparg
\\
\nonumber
&= u^1u^2u^3 \otimes \tparen{\nabla^\lc_{w^1}\nabla^\lc_{w^2}w^3 - \nabla^\lc_{w^2}\nabla^\lc_{w^1}w^3 }
\\
\nonumber
&\qquad - u^3\otimes \nabla^\lc_{(V^1 u^2)_\emparg w^2} w^3 + u^3\otimes \nabla^\lc_{(V^2 u^1)_\emparg w^1} w^3 - u^3\otimes\nabla^\lc_{[V^1_\emparg,V^2_\emparg]} w^3
\\
\nonumber
&= u^1u^2u^3 \otimes \tparen{\nabla^\lc_{w^1}\nabla^\lc_{w^2}w^3 - \nabla^\lc_{w^2}\nabla^\lc_{w^1}w^3 } - u^3\otimes\nabla^\lc_{u^1_\emparg u^2_\emparg [w^1,w^2]} w^3
\\
\nonumber
&\qquad - u^3\otimes \nabla^\lc_{u^1_\emparg\paren{(\mbfD_\emparg u^2) w^1}^\trid w^2} w^3 + u^3\otimes \nabla^\lc_{u^2_\emparg\paren{(\mbfD_\emparg u^1) w^2}^\trid w^1} w^3
\\
\nonumber
&= u^1u^2u^3 \otimes \tparen{\nabla^\lc_{w^1}\nabla^\lc_{w^2}w^3 - \nabla^\lc_{w^2}\nabla^\lc_{w^1}w^3 -\nabla^\lc_{[w^1,w^2]} w^3}
\\
\nonumber
&\qquad - u^1u^3 \tparen{(\mbfD_\emparg u^2)w^1}^\trid \otimes \nabla^\lc_{w^2}w^3+ u^2u^3\tparen{(\mbfD_\emparg u^1)w^2}^\trid \otimes \nabla^\lc_{w^1}w^3
\\
\nonumber
&= u^1u^2u^3 \otimes \tparen{R(w^1,w^2)w^3}
\\
\label{eq:t:Riemann:17}
&\qquad - u^1u^3 \tparen{(\mbfD_\emparg u^2)w^1}^\trid \otimes \nabla^\lc_{w^2}w^3+ u^2u^3\tparen{(\mbfD_\emparg u^1)w^2}^\trid \otimes \nabla^\lc_{w^1}w^3 \fstop
\end{align}

\paragraph{Mixed terms}
For the first term in~\eqref{eq:t:Riemann:3}, we have
\begin{align*}
\conn^\intr_{V^1}\tparen{(\boldOmega^\lc V^3)V^2} &= \conn^\intr_{V^1} (u^2u^3\otimes \nabla^\lc_{w^2} w^3) = \paren{\tparen{\mbfD_\emparg (u^2u^3)} V^1_\emparg}^\trid \otimes \nabla^\lc_{w^2}w^3
\\
&= u^1u^2 \paren{\tparen{\mbfD_\emparg u^3} w^1}^\trid \otimes \nabla^\lc_{w^2}w^3
\\
&\qquad + u^1u^3 \paren{\tparen{\mbfD_\emparg u^2} w^1}^\trid \otimes \nabla^\lc_{w^2}w^3 \fstop
\end{align*}

For the first term in~\eqref{eq:t:Riemann:4}, we have
\begin{align*}
\tparen{\boldOmega^\lc(\conn^\intr_{V^2}V^3)}V^1 &= \paren{\boldOmega^\lc\tparen{\tparen{(\mbfD_\emparg u^3)V^2_\emparg}^\trid\otimes w^3}} V^1= \tparen{(\mbfD_\emparg u^3)V^2_\emparg}^\trid \otimes \nabla^\lc_{V^1_\emparg} w^3
\\
&= u^1u^2 \tparen{(\mbfD_\emparg u^3)w^2_\emparg}^\trid \otimes \nabla^\lc_{w^1} w^3 \fstop
\end{align*}

Thus, combining all the necessary terms \emph{mutatis mutandis} we have
\begin{align*}
\eqref{eq:t:Riemann:3}+\eqref{eq:t:Riemann:4} + \eqref{eq:t:Riemann:17} = 0 \comma
\end{align*}
and the conclusion follows.
\end{proof}
\end{theorem}

Using the expression~\eqref{eq:t:RiemannTensor:0} for the Riemann tensor we may compute the sectional curvature at every pair of $\mbfG^\der_\mu$-orthogonal unit tangent vectors.

\begin{corollary}
Fix~$\mu\in\msP$.
Let~$w^1,w^2\in\mfX^\infty$ with~$\ttnorm{w^1}_\mu=\ttnorm{w^2}_\mu =1$ and~$\ttscalar{w^1}{w^2}_\mu=0$, and set~$V^j\eqdef \car\otimes w^j$ for~$j=1,2$.
Then, the sectional curvature~$\mbfK^\der_\mu$ of~$\msP$ at~$\mu$ in the plane spanned by~$V^1,V^2$ satisfies
\[
\mbfK^\der_\mu(V^1,V^2) = \int_M K(w^1,w^2) \paren{\ttabs{w^1}_g^2\, \ttabs{w^2}_g^2-g(w^1,w^2)^2} \diff\mu \fstop
\]
\end{corollary}

The corollary shows that if~$M$ has non-negative sectional curvature, then so does~$\msP$.
Furthermore, all considerations in~\cite[\S5, Rmk.~3, p.~434]{Lot07} concerning lower bounds for the sectional curvature of~$\msP$ remain valid.

\begin{remark}[Comparison with~{\cite{Lot07,Stu26}}]
The corollary should be compared with the analogous statement~\cite[Cor.~1]{Lot07}, very recently rederived in a rigorous metric way in~\cite{Stu26}.
In both these references, an additional correction term appears, accounting for the correction terms in the `reduced Riemann tensor'~$\mbfR^\grad$ which we address thoroughly in the next section,~\S\ref{ss:RiemannGrad}.
\end{remark}

We denote again by~$\conn^\lc$ the connection on every bundle associated to~$\VC{\infty}{b}$ induced by the Levi-Civita connection~$\conn^\lc$ on~$\VC{\infty}{b}$.

\begin{proposition}[Symmetries and Bianchi identities]\label{p:RiemannIdentities}
The Riemann tensor satisfies the \emph{symmetry conditions}
\begin{align}
\label{eq:p:RiemannIdentities:0}
\mbfR^\der(V^1,V^2) &= -\mbfR^\der(V^2,V^1) \comma
\\
\label{eq:p:RiemannIdentities:1}
\mbfG^\der\tparen{\mbfR^\der(V^1,V^2)V^3,V^4} &= - \mbfG^\der \tparen{\mbfR^\der(V^1, V^2)V^4,V^3} \comma
\\
\label{eq:p:RiemannIdentities:2}
\mbfG^\der\tparen{\mbfR^\der(V^1,V^2)V^3,V^4} &= \mbfG^\der\tparen{\mbfR^\der(V^3,V^4)V^1,V^2} \comma
\end{align}
and the \emph{Bianchi identities}
\begin{align}
\label{eq:p:RiemannIdentities:3}
\mbfR^\der(V^1,V^2)V^3 + \mbfR^\der(V^2,V^3)V^1 + \mbfR^\der(V^3,V^1)V^2 &= 0\comma
\\
\label{eq:p:RiemannIdentities:4}
(\conn^\lc_{V^1}\mbfR^\der)(V^2,V^3) + (\conn^\lc_{V^2}\mbfR^\der)(V^3,V^1) + (\conn^\lc_{V^3}\mbfR^\der)(V^1,V^2) &= 0 \fstop 
\end{align}

\begin{proof}
The proofs of~\eqref{eq:p:RiemannIdentities:0}-\eqref{eq:p:RiemannIdentities:3} follow from the analogous properties of the Riemann tensor~$R$ on~$M$.
The differential Bianchi identity may be proved by adapting an algebraic proof of the same identity for~$R$, see e.g.~\cite[\S2.3.1, Prop.~4, pp.~33f.]{Pet06}.
\end{proof}
\end{proposition}

\subsection{Reduced Riemann tensor}\label{ss:RiemannGrad}
Here, we compute the `reduced' Riemann tensor~$\mbfR^\grad$ and show that it coincides with the one derived in~\cite{Lot07}.
Throughout it will be convenient to (heuristically) regard~$\VC{\infty}{b,\grad}$ as an isometrically embedded subbundle of~$\VC{\infty}{b}$ with orthogonal projection~$\Pi$ onto the \emph{gradient-type subbundle}.
While our computations are rigorous, the understanding of some operators as the infinite-dimensional counterparts to standard ones in Riemannian geometry is heuristic in that they do not take values in the right spaces; in particular, they do not preserve smoothness.
We comment precisely on this fact in Remarks~\ref{r:Heur1} and~\ref{r:Heur2}.

\subsubsection{Levi-Civita connection of the gradient-type subbundle}
Denote by~$\conn^\grad\eqdef \Pi\conn^\lc$ the projection to~$\VC{\infty}{b,\grad}$ of the Levi-Civita connection~$\conn^\lc$ for~$\mbfG^\der$.

\begin{lemma}\label{l:LeviCivitaGrad}
$\conn^\grad$ can be heuristically regarded as the Levi-Civita connection for~$\mbfG^\grad$ on~$\VC{\infty}{b,\grad}$.
\end{lemma}

\begin{remark}\label{r:Heur1}
We stress that~$\conn^\grad$ is not, in general,~$\VC{\infty}{b,\grad}$-valued, due to the discontinuity of~$\Pi^\perp_\mu$ and to the fact that~$\Pi^\perp_\mu$ does not preserve smoothness unless~$\mu\in\msP^\infty$.
In the following, we will therefore interpret~$\conn^\grad$ pointwise.
This interpretation is sufficient for our purposes since:
\begin{itemize}
\item the metric tensor~$\mbfG^\der_\mu$ naturally extends from~$\mfX^\infty$ to~$\mfX_\mu$; 
\item the metric tensor~$\mbfG^\grad_\mu$ naturally extends from~$\mfX^\infty_\grad$ to its~$\mfX_\mu$-closure~$T^\grad_\mu\msP$;
\item before taking the projection, the connection~$\conn^\lc$ is a true connection since~$\conn^\intr$ is so by Proposition~\ref{p:Connection} and by tensoriality of~$\boldOmega^\lc$.
Thus, since the projection~$\Pi$ is reabsorbed whenever~$\conn^\grad$ is paired under~$\mbfG^\der$ with a gradient-type vector field, $\conn^\grad$ behaves as the connection~$\conn^\lc$ everywhere on~$\VC{\infty}{b,\grad}$.
\end{itemize}
\end{remark}

\begin{proof}[Proof of Lemma~\ref{l:LeviCivitaGrad}]
According to the heuristic understanding of~$\Pi\conn^\lc$ as a connection on~$\VC{\infty}{b,\grad}$, it suffices to verify $\mbfG^\grad$-metric compatibility and torsion-freeness.
These are pointwise properties, which can be verified pointwise for~$\conn^\lc$ and subsequently projected, as we now show.

\paragraph{$\mbfG^\grad$-metric compatibility}
Fix~$V^1,\dotsc, V^3\in\VC{\infty}{b,\grad}$.
From~\eqref{eq:MetricCompatibleP}, and recalling that~$\Pi V=V$ for every~$V\in\VC{\infty}{b,\grad}$, we have
\[
V^1\mbfG^\der(\Pi V^2,\Pi V^3) = \mbfG^\der(\conn^\lc_{V^1} V^2,\Pi V^3) + \mbfG^\der(\Pi V^2, \conn^\lc_{V^1} V^3)\comma
\]
and therefore, since~$\Pi$ is an orthogonal projection, and by definition of~$\mbfG^\grad$,
\[
V^1\mbfG^\grad(V^2,V^3) = \mbfG^\grad(\Pi \conn^\lc_{V^1} V^2, V^3) + \mbfG^\grad(V^2, \Pi \conn^\lc_{V^1} V^3)\fstop
\]

\paragraph{Torsion-freeness}
By~\eqref{eq:r:IntrConnectionModKerDiv} we have, for every~$V^1,V^2\in\VC{\infty}{b,\grad}$,
\[
\Pi \mbfT^\intr(V^1,V^2) = \mbfT^\intr(\Pi V^1, \Pi V^2) = \mbfT^\intr(V^1,V^2) \fstop
\]
Combining this with the definition of~$\conn^\lc$,~\eqref{eq:p:Torsion:0} and~\eqref{eq:ConnectionPullback}, we have, for every~$V,Z\in\VC{\infty}{b,\grad}$ and every elementary-tensor $1$-form~$\Omega= u\otimes\omega \in \diff\OC{0}{\infty}{b}$,
\begin{align*}
\tparen{\Omega\,\mbfT^{\Pi\nabla^\lc}(V^1,V^2)}_\mu &= \tparen{\Omega\Pi \mbfT^\intr(V^1,V^2) + \Omega\Pi\tparen{(\boldOmega^\lc V^2)V^1} - \Omega\Pi\tparen{(\boldOmega^\lc V^1)V^2}}_\mu 
\\
&= \mu\tparen{\Omega_\mu\Pi_\mu \tparen{-\pbracket{V^1_\mu}{V^2_\mu}+ (\nabla^\lc_{V^1_\mu}V^2_\mu)- (\nabla^\lc_{V^2_\mu}V^1_\mu)}}
\\
&= (u^1u^2 u)_\mu \cdot \mu \tparen{\omega \Pi_\mu \tparen{-[w^1,w^2]+(\nabla^\lc_{w^1}w^2)-(\nabla^\lc_{w^2}w^1)}}
\\
&= (u^1u^2 u)_\mu \cdot \mu \paren{\omega \Pi_\mu \tparen{T^{\nabla^\lc}(w^1,w^2)}}
\\
&=0
\end{align*}
since~$\nabla^\lc$ is torsion-free.
\end{proof}

\subsubsection{Second fundamental form of the gradient-type subbundle}
By analogy with the case of Euclidean submanifolds, we define the \emph{second fundamental form} of~$\VC{\infty}{b,\grad}$ in~$\VC{\infty}{b}$ as the bilinear form
\[
\SFF(V,Z) \eqdef \Pi^\tperp \conn^\lc_V Z\comma \qquad V,Z\in\VC{\infty}{b,\grad} \comma
\]
where~$\Pi^\perp_\mu\eqdef \Id_\mu-\Pi_\mu$ and~$\Pi_\mu$ is the projection operator~\eqref{eq:Projection}.

\begin{remark}\label{r:Heur2}
We stress that~$\mu\mapsto \SFF(V,Z)$ is not, in general, an element of~$\VC{\infty}{b}$, due to the discontinuity of~$\Pi^\perp_\mu$ and to the fact that~$\Pi^\perp_\mu$ does not preserve smoothness unless~$\mu\in\msP^\infty$.
We only have~$\SFF(V,Z)_\mu\in\ker\div_\mu$ pointwise on~$\msP$.
In the following, we will therefore interpret~$\SFF$ pointwise.
This interpretation is sufficient for our purposes since:
\begin{itemize}
\item the metric tensor~$\mbfG^\der_\mu$ naturally extends from~$\mfX^\infty$ to~$\mfX_\mu$; 
\item the metric tensor~$\mbfG^\grad_\mu$ naturally extends from~$\mfX^\infty_\grad$ to its~$\mfX_\mu$-closure~$T^\grad_\mu\msP$;
\item the value at~$\mu$ of every connection~$\conn_VZ$ involved in the computations depends only on the value of~$V$ at~$\mu$ rather than on its values in a neighborhood of~$\mu$.
\end{itemize}
\end{remark}

\begin{lemma}[Properties of~$\SFF$]
The second fundamental form is an $\FC{\infty}{b}$-bilinear form, satisfying
\begin{equation}\label{eq:l:SFF:0}
\SFF(V,Z)= \Pi^\tperp \tparen{(\boldOmega^\lc Z)V} \comma \qquad V,Z\in\VC{\infty}{b,\grad} \comma
\end{equation}
and its antisymmetric part is the non-integrability of~$\VC{\infty}{b,\grad}$ in~$\VC{\infty}{b}$, viz.
\begin{equation}\label{eq:l:SFF:0.1}
\mbfA(V,Z)\eqdef\SFF(V,Z)-\SFF(Z,V) = \Pi^\tperp \bbracket{V}{Z}= \Pi^\tperp \pbracket{V}{Z} \comma \qquad V,Z\in\VC{\infty}{b,\grad}\fstop
\end{equation}

\begin{proof}
Throughout, let~$V,Z\in\VC{\infty}{b,\grad}$ be arbitrary gradient-type vector fields.
The $\FC{\infty}{b}$-linearity of~$\SFF$ in the first entry follows directly from that of~$\conn^\lc$ in the first entry.
In the second entry, for every~$v\in \FC{\infty}{b}$, we have~$\conn^\lc_V(vZ) = v \conn^\lc_VZ+ (Vv)Z$.
Since~$Z\in\VC{\infty}{b,\grad}$, we also have~$\Pi^\tperp Z=0$ pointwise, and therefore~$\SFF(V,vZ)=\Pi^\tperp \conn^\lc_V(vZ) = v \Pi^\tperp \conn^\lc_VZ= v\SFF(V,Z)$.

Note that~$\conn^\intr_VZ\in\VC{\infty}{b,\grad}$ by~\eqref{eq:r:IntrConnectionModKerDiv}, so that~$\Pi^\tperp\conn^\intr_VZ\equiv 0$.
Thus, Equation~\eqref{eq:l:SFF:0} follows by definition~\eqref{eq:TrueConnection} of~$\conn^\lc=\conn^\intr+\boldOmega^\lc$.

Since~$\conn^\lc$ is torsion-free by Theorem~\ref{t:LCConnection}\ref{i:t:LCConnection:1}, we have~$\boldnabla^\lc_VZ-\boldnabla^\lc_ZV=\bbracket{V}{Z}$.
Applying~$\Pi^\tperp$ on both sides, and replacing the intrinsic torsion with the right-hand side of~\eqref{eq:p:Torsion:0}, we get
\[
\Pi^\tperp\bbracket{V}{Z}= \Pi^\tperp \conn^\intr_VZ - \Pi^\tperp \conn^\intr_ZV + \Pi^\tperp \pbracket{V}{Z} = \Pi^\tperp \pbracket{V}{Z} \comma
\]
again since~$\Pi^\tperp\conn^\intr_VZ\equiv 0$.
\end{proof}
\end{lemma}

\subsubsection{Computation of the `reduced' Riemann tensor}
Let us now compute the Riemann tensor~$\mbfR^\grad$ of~$\VC{\infty}{b,\grad}$.
Again, our terminology `Riemann tensor' ought to be understood heuristically, in the same sense as in~\cite{Lot07}.

\begin{theorem}
For all~$V^i\in\VC{\infty}{b,\grad}$ with~$i=1,\dotsc,4$, we have
\begin{align}\label{eq:t:RiemannGrad:0}
\mbfG^\grad\tparen{\Pi \tparen{\mbfR^\der(V^1,V^2)V^3},V^4} & = V^1\mbfG^\grad \tparen{\conn^\grad_{V^2}V^3, V^4} - \mbfG^\grad\tparen{\conn^\grad_{V^2}V^3, \conn^\grad_{V^1}V^4}
\\
\nonumber
&\qquad - V^2\mbfG^\grad\tparen{\conn^\grad_{V^1}V^3, V^4} + \mbfG^\grad\tparen{\conn^\grad_{V^1}V^3, \conn^\grad_{V^2}V^4}
\\
\nonumber
&\qquad -\mbfG^\grad\tparen{\conn^\grad_{\Pi\bbracket{V^1}{V^2}}V^3,V^4}
\\
\nonumber
&\qquad - \mbfG^\der\tparen{\SFF(V^2,V^3),\SFF(V^1,V^4)} + \mbfG^\der\tparen{\SFF(V^1,V^3), \SFF(V^2,V^4)} 
\\
\nonumber
&\qquad -\mbfG^\grad\tparen{(\boldOmega^\lc V^3)\, \mbfA(V^1,V^2),V^4}
\fstop 
\end{align}

In particular, the Riemann tensor~$\mbfR^\grad$ satisfies
\begin{align*}
\mbfR^\grad(V^1,V^2)V^3 &= \Pi\tparen{\mbfR^\der(V^1,V^2)V^3} 
\\
&\qquad + \mbfG^{\der,\sharp}\paren{\mbfG^\der\tparen{\SFF(V^2,V^3),\SFF(V^1,\emparg)} - \mbfG^\der\tparen{\SFF(V^1,V^3), \SFF(V^2,\emparg)} }
\\
&\qquad + (\boldOmega^\lc V^3)\, \mbfA(V^1,V^2) \fstop
\end{align*}

\begin{proof}
Note that, since~$\mbfR^\der$ satisfies all formal identities of a Riemann tensor, we have, for every~$V^1,\dotsc, V^4\in\VC{\infty}{b}$,
\begin{align*}
\mbfG^\der\tparen{\mbfR^\der(V^1,V^2)V^3,V^4} &= V^1\mbfG^\der (\conn^\lc_{V^2}V^3,V^4) - \mbfG^\der(\conn^\lc_{V^2}V^3,\conn^\lc_{V^1}V^4)
\\
&\qquad - V^2\mbfG^\der(\conn^\lc_{V^1}V^3, V^4) + \mbfG^\der(\conn^\lc_{V^1}V^3, \conn^\lc_{V^2}V^4)
\\
&\qquad - \mbfG^\der(\conn^\lc_{\bbracket{V^1}{V^2}}V^3,V^4) \fstop
\end{align*}
Specializing the expression above to~$V^1,\dotsc, V^4\in\VC{\infty}{b,\grad}$, and recalling that~$\Pi V=V$ for every~$V\in\VC{\infty}{b,\grad}$, we have
\begin{align*}
\mbfG^\der&\tparen{\mbfR^\der(V^1,V^2)V^3, \Pi V^4}
\\
& = V^1\mbfG^\der (\conn^\lc_{V^2}V^3,\Pi V^4)
\\
&\qquad - \mbfG^\der\tparen{\Pi (\conn^\lc_{V^2}V^3),\Pi (\conn^\lc_{V^1}V^4)} - \mbfG^\der\tparen{\Pi^\tperp (\conn^\lc_{V^2}V^3),\Pi^\tperp (\conn^\lc_{V^1}V^4)}
\\
&\qquad - V^2\mbfG^\der(\conn^\lc_{V^1}V^3, \Pi V^4)
\\
&\qquad + \mbfG^\der\tparen{\Pi(\conn^\lc_{V^1}V^3), \Pi(\conn^\lc_{V^2}V^4)} +  \mbfG^\der\tparen{\Pi^\tperp(\conn^\lc_{V^1}V^3), \Pi^\tperp(\conn^\lc_{V^2}V^4)}
\\
&\qquad - \mbfG^\der(\conn^\lc_{\bbracket{V^1}{V^2}}V^3, V^4) \fstop
\end{align*}
Using that~$\Pi$ is an orthogonal projection,
\begin{align*}
\mbfG^\grad&\tparen{\Pi \tparen{\mbfR^\der(V^1,V^2)V^3},V^4}
\\
& = V^1\mbfG^\grad \tparen{\Pi(\conn^\lc_{V^2}V^3), V^4}
\\
&\qquad - \mbfG^\grad\tparen{\Pi (\conn^\lc_{V^2}V^3),\Pi (\conn^\lc_{V^1}V^4)} - \mbfG^\der\tparen{\Pi^\tperp (\conn^\lc_{V^2}V^3),\Pi^\tperp (\conn^\lc_{V^1}V^4)}
\\
&\qquad - V^2\mbfG^\grad\tparen{\Pi(\conn^\lc_{V^1}V^3), V^4}
\\
&\qquad + \mbfG^\grad\tparen{\Pi(\conn^\lc_{V^1}V^3), \Pi(\conn^\lc_{V^2}V^4)} +  \mbfG^\der\tparen{\Pi^\tperp(\conn^\lc_{V^1}V^3), \Pi^\tperp(\conn^\lc_{V^2}V^4)}
\\
&\qquad - \mbfG^\der(\conn^\lc_{\bbracket{V^1}{V^2}}V^3, V^4) \comma
\end{align*}
whence, applying~\eqref{eq:TrueConnection} (with~$\nabla=\nabla^\lc$),~\eqref{eq:r:IntrConnectionModKerDiv}, and again that~$\Pi V=V$ for every~$V\in\VC{\infty}{b,\grad}$,
\begin{align*}
\mbfG^\grad&\tparen{\Pi \tparen{\mbfR^\der(V^1,V^2)V^3},V^4}
\\
& = V^1\mbfG^\grad \tparen{\Pi(\conn^\lc_{V^2}V^3), V^4}
\\
&\qquad - \mbfG^\grad\tparen{\Pi (\conn^\lc_{V^2}V^3),\Pi (\conn^\lc_{V^1}V^4)}
\\
&\qquad - \mbfG^\der\tparen{\SFF(V^2,V^3),\SFF(V^1,V^4)}
\\
&\qquad - V^2\mbfG^\grad\tparen{\conn^\intr_{V^1}V^3 + \Pi\tparen{(\boldOmega^\lc V^3)V^1}, V^4}
\\
&\qquad + \mbfG^\grad\tparen{\Pi(\conn^\lc_{V^1}V^3), \Pi(\conn^\lc_{V^2}V^4)}
\\
&\qquad + \mbfG^\der\tparen{\SFF(V^1,V^3), \SFF(V^2,V^4)}
\\
&\qquad - \mbfG^\der(\conn^\lc_{\bbracket{V^1}{V^2}}V^3, V^4)
\\
& = V^1\mbfG^\grad \tparen{\conn^\grad_{V^2}V^3, V^4} - \mbfG^\grad\tparen{\conn^\grad_{V^2}V^3, \conn^\grad_{V^1}V^4}
\\
&\qquad - V^2\mbfG^\grad\tparen{\conn^\grad_{V^1}V^3, V^4} + \mbfG^\grad\tparen{\conn^\grad_{V^1}V^3, \conn^\grad_{V^2}V^4}
\\
&\qquad - \mbfG^\der\tparen{\SFF(V^2,V^3),\SFF(V^1,V^4)} + \mbfG^\der\tparen{\SFF(V^1,V^3), \SFF(V^2,V^4)}
\\
&\qquad -\mbfG^\grad\tparen{\Pi\conn^\lc_{\bbracket{V^1}{V^2}}V^3,V^4} \fstop
\end{align*}

Let us proceed and compute the last term,~$\mbfG^\grad\tparen{\Pi\conn^\lc_{\bbracket{V^1}{V^2}}V^3,V^4}$.
We have
\begin{align*}
\mbfG^\der(\conn^\lc_{\bbracket{V^1}{V^2}}V^3, V^4) &= \mbfG^\der(\conn^\lc_{\Pi\bbracket{V^1}{V^2}}V^3, V^4) + \mbfG^\der(\conn^\lc_{\Pi^\tperp\bbracket{V^1}{V^2}}V^3, V^4)
\\
&= \mbfG^\grad\tparen{\Pi(\conn^\lc_{\Pi\bbracket{V^1}{V^2}}V^3), V^4} + \mbfG^\der\tparen{\Pi(\conn^\lc_{\Pi^\tperp\pbracket{V^1}{V^2}}V^3), V^4}
\end{align*}
by~\eqref{eq:l:SFF:0.1}.
Since~$\Pi^\tperp \pbracket{V^1}{V^2}$ is purely $T_\mu^\grad\msP$-orthogonal, we have~$\conn^\intr_{\Pi^\tperp\pbracket{V^1}{V^2}} V^3\equiv 0$ and thus
\[
\mbfG^\der\tparen{\Pi(\conn^\lc_{\Pi^\tperp\pbracket{V^1}{V^2}}V^3), V^4}= \mbfG^\der\tparen{\Pi\tparen{(\boldOmega^\lc V^3)\, \Pi^\tperp\!\pbracket{V^1}{V^2}},V^4} \fstop
\]
Combining the above equalities, we finally conclude that
\begin{align*}
\mbfG^\grad\tparen{\Pi \tparen{\mbfR^\der(V^1,V^2)V^3},V^4} & = V^1\mbfG^\grad \tparen{\conn^\grad_{V^2}V^3, V^4} - \mbfG^\grad\tparen{\conn^\grad_{V^2}V^3, \conn^\grad_{V^1}V^4}
\\
&\qquad - V^2\mbfG^\grad\tparen{\conn^\grad_{V^1}V^3, V^4} + \mbfG^\grad\tparen{\conn^\grad_{V^1}V^3, \conn^\grad_{V^2}V^4}
\\
&\qquad -\mbfG^\grad\tparen{\conn^\grad_{\Pi\bbracket{V^1}{V^2}}V^3,V^4}
\\
&\qquad - \mbfG^\der\tparen{\SFF(V^2,V^3),\SFF(V^1,V^4)} + \mbfG^\der\tparen{\SFF(V^1,V^3), \SFF(V^2,V^4)} 
\\
&\qquad -\mbfG^\grad\tparen{(\boldOmega^\lc V^3)\, \mbfA(V^1,V^2),V^4}
\comma
\end{align*}
which is precisely~\eqref{eq:t:RiemannGrad:0}.
\end{proof}
\end{theorem}

\begin{remark}[Comparison with~\cite{Lot07}]
Our definition of the `reduced' Riemann tensor~$\mbfR^\grad$ on~$\VC{\infty}{b,\grad}$ extends that of~\cite{Lot07}, i.e.\ the former coincides with the latter when restricted to constant vector fields of gradient type, at least at every point~$\mu=\rho\dvol_g$ with a $\Cinfty$ density~$\rho>0$.
\begin{proof}
Note that, for every~$V^i\in\VC{\infty}{b,\grad}$, with~$i=1,\dotsc,4$,
\begin{align}
\label{eq:r:Lott3:01}
\mbfG^\der\tparen{\SFF(V^2,V^3),\SFF(V^1,V^4)} &= g\tparen{\Pi^\tperp_{\emparg} \nabla^\lc_{V^2_{\emparg}}V^3_{\emparg}, \Pi^\tperp_{\emparg}\nabla^\lc_{V^1_{\emparg}}V^4_{\emparg}}^\trid \comma
\\
\label{eq:r:Lott3:02}
\mbfG^\der\tparen{\SFF(V^1,V^3), \SFF(V^2,V^4)} &= g\tparen{\Pi^\tperp_{\emparg} \nabla^\lc_{V^1_{\emparg}}V^3_{\emparg}, \Pi^\tperp_{\emparg}\nabla^\lc_{V^2_{\emparg}}V^4_{\emparg}}^\trid \comma
\\
\label{eq:r:Lott3:03}
\mbfG^\grad\tparen{(\boldOmega^\lc V^3)\, \mbfA(V^1,V^2),V^4} &= g\tparen{ \nabla^\lc_{\Pi^\perp_\emparg [V^1_\emparg,V^2_\emparg]} V^3_\emparg, V^4_\emparg}^\trid \fstop
\end{align}

Let us first address the terms in~\eqref{eq:r:Lott3:01}-\eqref{eq:r:Lott3:02}.
For~$i=1,\dotsc,4$, let~$\phi_i\in\Cinfty$, and choose~$\mu=\rho\dvol_g$ with~$\rho\in\Cinfty$ and~$\rho>0$.
Recall from~\cite[Eqn.~(5.1)]{Lot07} the definition of the operator
\begin{equation}\label{eq:r:Lott3:04}
T_{\phi\phi'} \eqdef (1-\Pi_\rho) \tparen{\nabla^i\phi\nabla_i\nabla_j \phi' \diff x^j} \comma
\end{equation}
where~$\Pi_\rho$ is the $L^2_\mu$-orthogonal projection operator onto (the closure of)~$\im(\diff)$, the image of the exterior differential~$\diff$, and thus it acts on $1$-forms; cf.~\cite[p.~430, after Lem.~3]{Lot07}.
(The left-hand side in~\eqref{eq:r:Lott3:04} depends on~$\rho$, but we suppress this dependence for simplicity of notation.)
Letting~$V^i\eqdef V_{\phi_i}$ and~$w^i\eqdef \nabla\phi_i$, with~$i=1,\dotsc,4$, be as in~\cite{Lot07}, we need to show that
\begin{equation}\label{eq:r:Lott3:1}
T_{\phi_1\phi_2} = g^\flat\tparen{\Pi^\tperp_\mu \nabla^\lc_{w^1}w^2}\comma
\end{equation}
where~$g^\flat\eqdef (g^\sharp)^{-1}\colon \mfX^\infty\to \Omega^\infty_1$ is the standard dual musical isomorphism.
By definition of~$1-\Pi_\rho$, we are left with showing that~$g^\flat(\nabla^\lc_{w^1}w^2) = \nabla^i\phi_1\nabla_i\nabla_j \phi_2 \diff x^j$ which, in fact, we have already shown in the proof of Remark~\ref{r:RemarkLott2}.
The equality in~\eqref{eq:r:Lott3:1} settles the last two terms in~\cite[Eqn.~(5.3)]{Lot07}.

Let us now address the term in~\eqref{eq:r:Lott3:03}.
Since~$V^1,V^2$ are constant vector fields, we have
\[
\Pi^\tperp_\mu [w^1, w^2] = \Pi^\tperp_\mu\tparen{ \nabla^\lc_{w^1} w^2 - \nabla^\lc_{w^2} w^1} = g^\sharp\tparen{T_{\phi_2\phi_1}-T_{\phi_1\phi_2}}
\]
and therefore, since~$T_{\phi\phi'}$ is antisymmetric by~\cite[Lem.~6, p.~432]{Lot07},
\[
\Pi^\tperp_\mu [w^1, w^2] = -2\, g^\sharp(T_{\phi_1\phi_2}) \fstop
\]
Note that~$Z\eqdef g^\sharp(T_{\phi_1\phi_2})\in \ker\div_\mu$, since~$T_{\phi\phi'}\in \im(\diff)^\tperp$.

Now, set~$h\eqdef g(w^3,w^4)\in\Cinfty$. By metric compatibility of~$\nabla^\lc$,
\begin{align}\label{eq:r:Lott3:11}
Zh = Z g(w^3,w^4) = g(\nabla^\lc_Z w^3,w^4)+ g(w^3,\nabla^\lc_Zw^4) \fstop
\end{align}
For the last term, we have, by symmetry of~$\Hess_g(\phi_4)$,
\begin{align}\label{eq:r:Lott3:12}
g(w^3,\nabla^\lc_Zw^4) = \Hess_g(\phi_4) (Z,w^3) = \Hess_g(\phi_4) (w^3,Z) = g(\nabla^\lc_{w^3}w^4,Z)\fstop
\end{align}
Combining~\eqref{eq:r:Lott3:11} and~\eqref{eq:r:Lott3:12} we thus obtain
\[
g(Z,\nabla h)= Zh= g(\nabla^\lc_Z w^3,w^4)+g(Z,\nabla^\lc_{w^3}w^4) \fstop
\]
Integrating w.r.t.~$\mu$, and since~$Z\in \ker\div_\mu$, we conclude that the integral of the left-hand side above vanishes, and therefore,
\[
\int g(\nabla^\lc_Z w^3,w^4)\diff \mu = - \int g(Z,\nabla^\lc_{w^3}w^4)\diff\mu = -\int g(T_{\phi_1\phi_2} T_{\phi_3\phi_4}) \diff\mu \fstop
\]
Substituting the definition of~$Z$ then shows that the term in~\eqref{eq:r:Lott3:03} satisfies
\[
\mbfG^\grad\tparen{(\boldOmega^\lc V^3)\, \mbfA(V^1,V^2),V^4} = -2\,\mbfG^\der (T_{\phi_1\phi_2}, T_{\phi_3\phi_4}) \fstop
\]
Replacing all terms in~\eqref{eq:t:RiemannGrad:0} finally proves the assertion.
\end{proof}
\end{remark}

\end{document}